\documentclass[11pt]{article}
\usepackage[utf8]{inputenc}
\usepackage[T1]{fontenc}
\usepackage{lmodern}
\usepackage{amsmath,amssymb}
\usepackage{amsthm,amsfonts}
\usepackage{fullpage}
\usepackage{hyperref}
\usepackage{authblk}
\usepackage[usenames,dvipsnames]{xcolor}
\usepackage[shortlabels]{enumitem}
\usepackage{tikz}

\newtheorem{theorem}{Theorem}
\newtheorem{definition}[theorem]{Definition}
\newtheorem{lemma}[theorem]{Lemma}
\newtheorem{corollary}[theorem]{Corollary}
\newtheorem{conjecture}[theorem]{Conjecture}
\newtheorem{proposition}[theorem]{Proposition}

\newcommand\EE{\mathbb{E}}

\begin{document}

\date{}
\author[1]{Jesse Campion Loth}
\author[1]{Kevin Halasz}
\author[1]{Tomáš Masařík\thanks{T.M.~was supported by a postdoctoral fellowship at the Simon Fraser University through NSERC grants R611450 and R611368.}}
\author[1]{\\Bojan Mohar\thanks{B.M.~was supported in part by the NSERC Discovery Grant R611450 (Canada) and by the Research Project J1-8130 of ARRS (Slovenia).}}
\author[2]{Robert Šámal\thanks{
R.S. was partially supported by grant 19-21082S of the Czech Science Foundation.
This project has received funding from the European Research Council
(ERC) under the European Union’s Horizon 2020 research and innovation
programme (grant agreement No 810115). 
This project has received funding from the European Union’s Horizon 2020
research and innovation programme under the Marie Skłodowska-Curie grant
agreement No 823748.}}

\affil[1]{Department of Mathematics, Simon Fraser University, Burnaby, BC, V5A 1S6, Canada\\
\texttt{\{jcampion, khalasz, tmasarik, mohar\}@sfu.ca}}

\affil[2]{Computer Science Institute, Faculty of Mathematics and Physics, Charles University, Praha, 118 00, Czech Republic\\
\texttt{samal@iuuk.mff.cuni.cz}}

\title{Random 2-cell embeddings of multistars}

\maketitle

\begin{abstract}
Random 2-cell embeddings of a given graph $G$ are obtained by choosing a random local rotation around every vertex.
  We analyze the expected number of faces, $\EE[F_G]$, of such an embedding which is equivalent to studying its average genus.
 So far, tight results are known for two families called monopoles and dipoles.
  We extend the dipole result to a more general family called multistars, i.e., loopless multigraphs in which there is a vertex incident with all the edges.  
    In particular, we 
    show that the expected number of faces of every multistar with $n$ nonleaf edges lies in an interval of length $2/(n+1)$ centered at the expected number of faces of an $n$-edge dipole.
  This allows us to derive bounds on $\EE[F_G]$ for any given graph $G$ in terms of vertex degrees.
  We conjecture that $\EE[F_G]\le O(n)$ for any simple $n$-vertex graph $G$. %
\end{abstract}

\section{Introduction}

By an \emph{embedding} of a graph $G$, we mean a 2-cell embedding of $G$ in some orientable surface. Two embeddings of $G$ are \emph{equivalent} if there is an orientation-preserving homeomorphism of the surface mapping the graph in one embedding onto the graph in the other, and the restriction of the homeomorphism to the graph is the identity isomorphism. Equivalent embeddings are considered the same since they define the same \emph{map}, where a map is considered as the incidence structure of vertices, edges, and faces of the embedding. It is well known \cite{MT01,Wh73} that equivalence classes of 2-cell embeddings of $G$ (i.e., maps whose underlying graph is $G$) are in bijective correspondence with \emph{local rotations}, where for each vertex $v\in V(G)$ we prescribe a cyclic permutation $\pi_v$ of the half-edges, or \emph{darts}, incident with $v$.

We consider the ensemble of all maps of $G$ endowed with the uniform probability distribution. The genus and the number of faces of a random map of $G$ become random variables in this setting. This gives rise to the notion of the \emph{average genus} of the graph and leads to \emph{random topological graph theory} as termed by White \cite{Wh94}.  A monograph by Lando and Zvonkin \cite{LandoZvonkin} provides a thorough treatment of various applications of random graph embeddings, ranging from graph theory and combinatorics to abstract algebra and theoretical physics. 

Two special cases of random embeddings are well understood. The first one is when the graph is a bouquet of $n$ loops (also called a \emph{monopole}), which is the graph with a single vertex and $n$ loops incident with the vertex; see \cite{Ch11,Gross89,Ja87,Jackson1994_integral,Za95}. By duality, the maps of the monopole with $n$ loops correspond to unicellular maps \cite{Ch11} with $n$ edges. The second well-studied case is the \emph{$n$-dipole}, a two-vertex graph with $n$ edges joining the two vertices; see \cite{Andrews1994,Chen2020,CR16,CMS12,Jackson1994_algebraic,Jackson1994_integral,KL93,RiThesis,St11}.

Here we consider the more general family of graphs, called \emph{multistars}.
These are loopless multigraphs in which there is a vertex incident with all the edges. Formally, we have one center vertex $v_0$ incident with $n$ edges; these edges lead to $k\ge1$ other vertices, $v_1,\dots,v_k$, with $n_i$ edges between $v_0$ and $v_i$ ($1\le i\le k$). As $n_1\ge n_2\ge \cdots \ge n_k\ge 1$, where $\sum_{i=1}^k n_i = n$, we see that multistars with $n$ edges are in bijective correspondence with partitions of $n$. 
An expression for the genus polynomials of multistars was obtained by Stanley \cite{St11} in the language of products of permutation in conjugacy classes.
Our main results use this formulation to derive precise bounds for the expected genus of these graphs (see Section~\ref{sect:multistars}).

It is important to note that we are not merely interested in multistars as another family of graphs whose genus distribution may be analyzed. Rather, we are primarily motivated by the fact that multistars are useful in the study of random embeddings of arbitrary graphs. We explain why this is so in Section \ref{sect:adding a vertex}; see also the overview in Section \ref{sect:our results} below.

Although most previous works in random topological graph theory concern the (average) genus, we decided to consider the (average) number of faces, which is an equivalent quantity by the Euler-Poincar\'{e} formula, but gives more appealing statements.

\subsection{Our results}
\label{sect:our results}

The paper is organized as follows.
In Section~\ref{sec:dipole}, we show that the expected number of faces for a random embedding of a dipole with $n$ edges is precisely $H_{n-1} + \left\lceil \frac{n}{2} \right\rceil^{-1}$, where $H_{n-1} = 1+\tfrac{1}{2}+\tfrac{1}{3}+\cdots + \tfrac{1}{n-1}$ is the harmonic sum (see Corollary~\ref{cor:dipoleE}). Previously, Stahl \cite{Stahl1995} proved that the average number of faces is at most $H_{n-1} + 1$. It is worth noting that we are able to obtain our exact result with a relatively short proof based on Stanley's generating function~\cite{St11}.

In Section~\ref{sect:multistars} we extend the dipole result to multistars, showing that they have the same expected number of faces as dipoles up to a difference of $\pm\tfrac{1}{n+1}$ (see Theorem~\ref{thm:asymmulti}). In Section~\ref{sect:adding a vertex} we note that the result about multistars can be used in a more general setting, where we consider a map to which we add a new vertex and consider the expected number of new faces added after doing so.  In particular, our Theorem~\ref{thm:newvert} shows that the expected number of new faces obtained when adding a new vertex of degree $d$ is at most $\log(d)+1$ (where we use $\log(\cdot)$ to denote the natural logarithm). This immediately implies an old result of Stahl \cite{St91Short} that the expected number of faces in a random embedding of an arbitrary graph of order $n$ is at most $n\log(n)$. 

We also exhibit a general construction of graph families with a linear number of expected faces. However, we were not able to find a family with a superlinear number of expected faces and we ended up proposing the following conjecture.

\begin{conjecture}\label{conj:simple}
For every $n$-vertex simple graph $G$, the expected number of faces when selecting an orientable embedding of $G$ uniformly at random is $O(n)$.
\end{conjecture}

We use multistars to obtain new upper bounds for the expected number of faces of several families of graphs on $n$ vertices.  A notable outcome is for $d$-regular graphs, where the conclusion is that the expected number of faces is at most $n\log(d)$. More generally, the same result works for $d$-degenerate graphs (see Theorem~\ref{thm:degbnd} and Corollary~\ref{cor:degen}). Both of these results can be considered as supporting evidence for Conjecture \ref{conj:simple}.

When we allow multiple edges, Conjecture \ref{conj:simple} has to be adjusted.

\begin{conjecture}
For every $n$-vertex multigraph $G$ with maximum edge-multiplicity $\mu$, the expected number of faces when selecting an orientable embedding of $G$ uniformly at random is $O(n \log (2\mu))$.\label{endingconj}
\end{conjecture}

Notice that the dipole, considered in Section~\ref{sec:dipole}, gives a family of graphs for which Conjecture~\ref{endingconj} is tight. Moreover, a long path in which every second edge is replaced by a dipole with $\mu\ge2$ edges gives a tight family in which each of $n$ and $\mu$ can independently tend to infinity.

\section{The dipole}
\label{sec:dipole}

Embeddings of monopoles and dipoles have connections to other areas. Thus it is not surprising that they have been extensively studies.
The genus distribution of the monopole can be traced back to a celebrated result of Harer and Zagier~\cite{HZ86}, who used matrix integrals in their proofs. Independently, Jackson \cite{Ja87} proved an analogous result by a different method, using the character theory of the symmetric group. However, it was Gross, Robbins, and Tucker~\cite{Gross89} who used Jackson's result to state the genus distribution of the monopole explicitly. Later, Zagier \cite{Za95} found another, shorter proof using character theory of the symmetric group.
A combinatorial proof was found later by Chapuy \cite{Ch11}, who used it in the enumeration of unicellular maps. 
The corresponding result for dipoles was given by Rieper \cite{RiThesis} in his PhD thesis. The dipole genus distribution was found independently by Kwak and Lee \cite{KL93} and also by Jackson~\cite{Jackson1994_integral}, who provided the genus distribution for both monopoles %
and dipoles %
using integral representations.
Shortly after, Andrews, Jackson, and Visentin %
\cite{Andrews1994} provided a parity-specific genus distribution.
In \cite{Jackson1994_algebraic}, Jackson gave an overview of the methods used to compute genus distributions of 2-cell embeddings on orientable as well as nonorientable surfaces.
The result for dipoles was later reproved by Zagier \cite{Za95} (using character theory) and Stanley \cite{St11} (a combinatorial proof using symmetric functions). Cori, Marcus, and Schaefer \cite{CMS12} found a first combinatorial proof. A generalized version (dipoles with loops) was considered by Goulden and Slofstra \cite{GS10} and Gross, Mansour, and Tucker \cite{GMT17}. Recently, Chen and Reidys~\cite{CR16}, and  Chen~\cite{Chen2020} gave another purely combinatorial proof.

Here we state an exact formula for the expected number of faces of dipoles (Corollary \ref{cor:dipoleE} below). Notably, even though the precise genus distribution of dipoles is known, we were not able to locate the explicit computation of the average genus or any mention of this result. The formula of Corollary~\ref{cor:dipoleE} is surprisingly nice and serves as a comparison to the results about multistars that we treat later.

Before proceeding, we outline some notation. Let $[n] := \{1,\dots,n\}$ and let $S_n$ be the symmetric group acting on the set $[n]$. Furthermore, let $C_n \subseteq S_n$ be the set of \emph{full cycles} of length $n$, i.e., the permutations in $S_n$ with precisely one cycle of length $n$.
For any real number $x$ and a positive integer $k$, we denote by $(x)_k = x(x-1)\cdots (x-k+1)$ the \emph{falling factorial} of $x$.
Let $c(n,k)$ be the unsigned Stirling number of the first kind, and $s(n,k)$ be the signed equivalent, such that $s(n,k) = (-1)^{n-k}c(n,k)$.  Let $H_n$ be the $n^{\rm th}$ harmonic number: $H_n = 1 + 1/2 + 1/3 + \dots + 1/n$. For convenience we also set $H_0=0$. The value of $H_n$ is asymptotically logarithmic: $\lim_{n \to +\infty} (H_n - \log(n)) = \gamma$ where $\gamma \approx 0.5772$ is the Euler-Mascheroni constant.

We will discuss random embeddings of the \emph{dipole} $D_n$: the graph with two vertices and $n$ parallel edges joining them.  Each embedding of $D_n$ is determined by the local rotations at both vertices. In this case, each local rotation is a full cycle in $C_n$.  This means there is a bijection between embeddings of $D_n$ and pairs $\{(\sigma, \tau): \sigma, \tau \in C_n\}$.  It is then fairly easy to see that the faces in an embedding given by $(\sigma, \tau)$ correspond to the cycles in the permutation product $\sigma \tau$.

Calculating the expected number of faces in an embedding of $D_n$ is therefore equivalent to calculating the expected number of cycles in a product of two full cycles taken randomly from $C_n$.  The labelling on the symbols in $S_n$ is arbitrary, so we may fix one of the full cycles to be $\sigma = (1 \, 2 \, 3 \, \dots \, n)$ and just consider the set $\{(\sigma, \tau): \tau \in C_n\}$.  Let $F$ be the random variable for the number of cycles in $\sigma \tau$ when $\tau$ is chosen uniformly at random from $C_n$.  Therefore the expected number of faces in a random embedding of $D_n$ is equal to $\EE[F]$.

The combinatorial problem of finding the number of cycles in a product of two full cycles has already been the object of extensive research.  First note that the permutation parity argument implies that the number of pairs $(\sigma, \tau) \in C_n\times C_n$ for which the product $\sigma\tau$ has $k$ cycles is zero when $k \not\equiv n \pmod 2$.  
Stanley \cite[Corollary 3.4]{St11} proved the following result.

\begin{theorem}[Stanley \cite{St11}]
\label{thm:dipoledistro}
    The number of cyclic permutations $\tau \in C_n$ such that the product 
    $(1 \, 2 \, 3 \, \dots \, n) \tau$ has $k$ cycles is equal to $\tfrac{2}{n(n+1)} c(n+1,k)$ if $n-k$ is even (and is zero if $n-k$ is odd).
\end{theorem}

Stanley's theorem yields the face distribution for random embeddings of the dipole. As shown in the proof of Corollary~\ref{cor:dipoleE} below, it gives a simple proof about the average genus of $D_n$ using only basic combinatorial techniques. 

\begin{corollary}\label{cor:dipoleE}
Let $F$ be the number of faces in a random embedding of $D_n$, where $n\ge2$. Then
$$
  \EE(F) = \begin{cases}
    H_{n-1} + \tfrac{2}{n}, & \text{if $n$ is even;} \\
    H_{n-1} + \tfrac{2}{n+1}, & \text{if $n$ is odd}.
    \end{cases}
$$
\end{corollary}

\begin{proof}
Suppose first that $n$ is even. Theorem~\ref{thm:dipoledistro} gives that 
\begin{equation} \EE(F) = \frac{2}{(n+1)!} \sum_{k \text{ even}} c(n+1,k) k. 
\label{eq:expF}
\end{equation}
To determine the value of this sum, we use the well-known generating functions for signed and unsigned Stirling numbers (see e.g.\! \cite[p.~942]{St11}) to obtain
\begin{equation}
\sum_{k \text{ even}} c(n+1,k) q^k = \sum_{k=1}^{n+1} \tfrac{1}{2}( c(n+1,k)-s(n+1,k)) q^k = \tfrac{1}{2}(  (q+n)_{n+1}-(q)_{n+1}).
\label{eq:evenstirling}
\end{equation}
Differentiating both sides of this equation and plugging in $q=1$, we obtain
 \[ 
 \sum_{k \text{ even}} k\, c(n+1,k) = 
 \tfrac{1}{2}\Bigl((n-1)! + \sum_{i=1}^{n+1} \frac{(n+1)!}{i} \Bigr) = \frac{(n+1)!}{2}\Bigl(\frac{1}{n(n+1)} + H_{n+1}\Bigr),
 \]
which we can plug in to \eqref{eq:expF} and simplify to get the desired result in the case when $n$ is even.

The case where $n$ is odd is similar, with the difference coming from the fact that  
\[\sum_{k \text{ odd}} c(n+1,k) q^k = \sum_{k=1}^{n+1} \tfrac{1}{2}( c(n+1,k)+s(n+1,k)) q^k .\]
Differentiating both sides of this equation, setting $q=1$, plugging the result into the analogous version of \eqref{eq:expF} and simplifying gives $\EE(F) = H_{n-1}+\frac{2}{n+1}$ in this case.
\end{proof}

\section{Multistars}
\label{sect:multistars}

As mentioned in the introduction, multistars with $n$ edges are in bijective correspondence with partitions of $n$. If $n=n_1+\cdots + n_k$, we denote the partition as $\lambda = (n_1,\dots,n_k)$ and write $\lambda \vdash n$. It is also customary to write $\lambda = n_1^{i_1}\dots n_r^{i_r}$ if $\lambda$ has $n_t$ repeated $i_t$ times, where we are allowed to leave out the repetitions of 1 (when $n$ is clear from the context). For example, $\lambda = (5,4,4,4,2,2,1,1,1) = 5\,4^3\,2^2\,1^3 = 5\,4^3\,2^2$.

We denote by $C_\lambda$ the set of all permutations of type $\lambda$.  Recall that the \emph{type} of a permutation in $S_n$ is the partition of $n$ corresponding to the cycle lengths of the permutation when written as the product of disjoint cycles.

Let $\lambda$ be a partition of $n$ with $k$ parts $\lambda_1,\dots,\lambda_k$. We consider the \emph{multistar of type $\lambda$}: the multistar with $k$ outer vertices and $\lambda_i$ edges from the central vertex to the $i^{th}$ outer vertex.  Call this $K_\lambda(n)$.  
Let $r(\lambda)$ denote the number of parts of size 1 in $\lambda$, and note that a vertex of degree 1 in a multistar does not affect the number of faces. 

The generating function for the number of cycles in products of permutations in certain conjugacy classes was found by Stanley \cite{St11} using symmetric functions. It can be expressed using the \emph{shift operator} $E$ that is defined by the rule $E(f(q)) = f(q-1)$.  For example $E((q)_t) = (q-1)_t$.

\begin{theorem}[\cite{St11}]\label{thm:St11}
    Let $f_{\lambda}(j)$ denote the number of permutations in $C_{\lambda}$, whose product with the full cycle $(1\,2\cdots n)$ is a permutation with $j$ cycles.  Then:
\[ 
   \sum_{j=1}^n f_{\lambda}(j) q^j = 
   \frac{|C_\lambda|}{(n+1)!} \left(\,\prod_{i=1}^k (1-E^{\lambda_i})\right) (q+n)_{n+1}. 
\]
\end{theorem}

We use this result to derive our main result of this section.

\begin{theorem} \label{thm:asymmulti}
    Let $F_\lambda(n)$ be the random variable denoting the number of faces in a random embedding of $K_{\lambda}(n)$ and let $n'=n-r(\lambda)$. Then
    \[ \mathbb{E}(F_\lambda(n)) \in \left( \Delta_{n'} - \tfrac{1}{n'+1} ,\, \Delta_{n'} + \tfrac{1}{n'+1}\right),\]
    where $\Delta_{n'} = H_{n'-1} + \left\lceil \tfrac{n'}{2} \right\rceil^{-1}$.
\end{theorem}

\begin{proof}[Proof of Theorem~\ref{thm:asymmulti}]
 We may assume $\lambda(n)$ has no parts of size 1, as otherwise, we can remove any vertices of degree one from the multistar without affecting the number of faces. Having this assumption, we can use $n$ instead of $n'$. 
 We shall write $\mu\preceq \lambda$ for \emph{subpartitions} $\mu$ of $\lambda$, meaning that $\mu$ can be obtained from $\lambda$ by omitting some of the terms $\lambda_i$. We also use the notation $l(\lambda) = k$ for the number of parts of the partition and denote by $|\lambda|=\lambda_1+\cdots+\lambda_k$ the \emph{weight} of $\lambda$. This notation extends to the subpartitions of $\lambda$.

 Using the notation from Theorem~\ref{thm:St11}, we need to estimate 
 $\EE{F_\lambda(n)} = \frac{1}{|C_\lambda|} \sum_{j=1}^n j f_\lambda(j) $. We will proceed in a similar manner to the proof of
 Corollary~\ref{cor:dipoleE}, but using Theorem~\ref{thm:St11} instead of equation~(\ref{eq:evenstirling}). 
 Let us first note that 
  \begin{align*}
      \frac{(n+1)!}{|C_\lambda|} \sum_{j=1}^n f_\lambda(j) q^j &= 
      \left(\,\prod_{i=1}^k (1-E^{\lambda_i})\right) (q+n)_{n+1} \\
      &= \sum_{\mu \preceq \lambda} (-1)^{l(\mu)} E^{\vert \mu \vert} (q+n)_{n+1} \\
      &= \Bigl( (q+n)_{n+1} + (-1)^{l(\lambda)} (q)_{n+1} \Bigr)+ \sum_{\mu \precsim \lambda} (-1)^{l(\mu)} E^{\vert \mu \vert} (q+n)_{n+1} \\
      &= \sum_{k=1}^{n+1} \left( c(n+1,k) + (-1)^{l(\lambda)}s(n+1,k)\right)q^k+ \sum_{\mu \precsim \lambda} (-1)^{l(\mu)} E^{\vert \mu \vert} (q+n)_{n+1}
  \end{align*}
 where $\mu \precsim \lambda$ means that $\mu$ is a nonempty, nontrivial subpartition of $\lambda$. The first sum on the last line above is the same as in \eqref{eq:evenstirling} (or the corresponding odd result), so differentiating and plugging in $q=1$ gives the expected number of faces in a random embedding of the $n$-dipole. 
 Hence, the expected number of faces in a random embedding of $K_\lambda(n)$ is:
 \begin{align}
 \begin{split}  
     \Delta_n + \frac{1}{(n+1)!} \sum_{\mu \precsim \lambda} (-1)^{l(\mu)} \sum_{j \geq 1} j [q^j]  (q + n - \vert \mu \vert)_{n+1} .
 \end{split}
     \label{Bigidea}
 \end{align}
To complete the proof, it suffices to see that the absolute value of the last term in (\ref{Bigidea}) is less than~$\tfrac{1}{2}$.
To prove this, we take the derivatives of the falling factorial terms and evaluate them at $q=1$ %
to obtain the following expression:
 \begin{align*}
    \sum_{j \geq 1} j\, [q^j]  (q+n - \vert \mu \vert)_{n +1} 
     &= \Bigl[\frac{d}{dq}(q+n - \vert \mu \vert)_{n +1}\Bigr]_{q=1} \\
     &= (n - \vert \mu \vert + 1)! (\vert \mu \vert - 1)! (-1)^{\vert \mu \vert - 1}.
 \end{align*}
 
Putting this into the last term in (\ref{Bigidea}), and denoting by $\theta$ its absolute value, we obtain:
\begin{align*}
  \theta &= \biggl| \frac{1}{(n+1)!} \sum_{\mu \precsim \lambda} (-1)^{l(\mu)} (n - \vert \mu \vert + 1)! (\vert \mu \vert - 1)! (-1)^{\vert \mu \vert - 1}\biggr| \\
  &\le \frac{1}{n+1} \sum_{\mu \preceq \lambda} \frac{1}{\binom{n}{\vert \mu \vert - 1}} \\
  &= \frac{1}{n+1} \sum_{i=2}^{n-2} \frac{p_i}{\binom{n}{i - 1}}
 \end{align*}
where $p_i$ is the number of subpartitions of $\lambda$ whose sum is $i$.

To bound the terms $p_i\binom{n}{i - 1}^{-1}$,
we split each $\lambda_i \in \lambda$ into parts of size $2$ and possibly one part of size $3$ (if $\lambda_i$ is odd), then combine these new parts to obtain a finer partition $\lambda'$. 
Observe that the number of subpartitions of $\lambda'$ whose weight is $i$ is greater than the number of weight $i$ subpartitions of $\lambda$. Since $p_2+p_3\le n/2$ for $\lambda'$, this implies that
\begin{equation}
  \frac{p_2}{\binom{n}{1}} +  \frac{p_3}{\binom{n}{2}} \leq \frac{1}{2}.
\label{eq:bound for i=2,3}    
\end{equation}
For the case when $4\le i \le \frac{n}{2}$, we obtain a bound in the following way.
Each $\mu \precsim \lambda$ with $|\mu| = i$ corresponds to a $\mu' \precsim \lambda'$ with $|\mu'|=i$ and, using several parts of size 0 if necessary, $l(\mu') = \lfloor \frac{i}{2} \rfloor$. This implies that
\[ p_i \leq \binom{\lfloor n/2\rfloor}{\lfloor i/2\rfloor}.\]%
Thus, if $i$ is even and between 4 and $\frac{n}{2}$, we have
\begin{align*}
    \frac{p_i}{\binom{n}{i-1}} &\leq \frac{n(n-2) \cdots (n-i+2)}{n(n-1)\cdots(n-i+2)} \frac{(i-1)(i-2)\cdots 2}{i(i-2)\cdots 2} \\
    &=  \frac{1}{i} \frac{ i-1}{n-1} \frac{i-3}{n-3} \cdots \frac{3}{n-i+3}\\
    &\leq \frac{1}{2(n-1)}.
\end{align*}

A similar argument then gives, for $i$ odd and between 5 and $\frac{n}{2}$, the bound \[p_i\binom{n}{i-1}^{-1} \leq \frac{1}{(n-1)} \frac{ i-2}{n-3} \frac{i-4}{n-5} \cdots \frac{3}{n-i+2} \leq \frac{1}{2(n-1)}.\]
    
As the complement of any subpartition of $\lambda$ with sum $i$ is a subpartition with sum $n-i$, we have $p_i = p_{n-i}$. This implies that $p_{n-2}\binom{n}{n-3}^{-1} \leq \frac{3}{(n-1)(n-2)}$ and $p_{n-3}\binom{n}{n-4}^{-1} \leq \frac{8}{(n-1)(n-2)(n-3)}$, where both of these terms are less than $\frac{1}{2(n-1)}$ for $n \geq 8$. For $i \geq 4$, we can use the fact that $\binom{n}{i-1} \leq \binom{n}{i+1}$ for $i \leq \frac{n}{2}$ to conclude that:
\[ 
  \frac{p_{n-i}}{\binom{n}{n-i-1}} \leq \frac{p_i}{\binom{n}{i+1}} \leq \frac{1}{2(n-1)}. 
\]

Putting this all together gives:
\[\theta \leq \frac{1}{n+1} \sum_{i=2}^{n-2} \frac{p_i}{\binom{n}{i - 1}}
\leq \frac{1}{n+1} \left( \frac{1}{2} + (n-5)\frac{1}{2(n-1)} \right)
< \frac{1}{n+1}. \qedhere
\]

\end{proof}

Let us also observe that the multistar result in particular applies to dipoles and monopoles. 
The dipoles correspond to $K_\lambda(n)$ with $\lambda = n^1$.
Recall that the value $\Delta_{n'} = H_{n'-1} + \left\lceil \tfrac{n'}{2} \right\rceil^{-1}$ in Theorem~\ref{thm:asymmulti} is precisely the same as the expected number of faces for the dipole with $n'$ edges in Corollary~\ref{cor:dipoleE}.
The multistar $K_\lambda(2n)$ with $\lambda = 2^n$ is isomorphic to 
the monopole with $n$ edges, each of them subdivided.  
Therefore, $K_{2^n}(2n)$ and the monopole with $n$ edges have the same
distribution of the number of faces.
This special case of Theorem \ref{thm:asymmulti} can be compared with a result of Stahl \cite[Theorem 2.5]{St90}, who proved that the difference between $H_{2n}$ and the expected number of faces in a random embedding of the monopole with $n$ loops approaches zero as $n$ grows arbitrarily large.

\section{General graphs}
\label{sect:adding a vertex}

We will use the results of the previous sections to obtain upper bounds on the expected number of faces for random embeddings of more general classes of graphs. As an embedding of a graph on an orientable surface is uniquely determined by the local rotations at the vertices, we formally represent an embedding of a graph $G$ as a set $\Pi = \{\pi_v \mid v \in V(G)\}$, where $\pi_v$ is a cyclic permutation of the darts incident to $v$. Moreover, we may draw such an embedding in the plane by  connecting the corresponding darts in the correct order. The edges may, of course, cross each other in the drawing.  We can then trace out the faces in an embedding by following along one side of a path of edges. This process is shown in Figure~\ref{k4}: The embedding of $K_4$ represented in the figure has two faces, traced out in blue and red dashed lines.  For our purpose, it will be easier to draw the embedding in a different way: we will draw a circle for each face in the embedding then put the vertices around the perimeter of this circle in the order they appear on the facial walk. Of course, any vertex of degree $d$ appears altogether $d$ times and can appear in the same facial walk multiple times. Formally, the circles in this drawing correspond to the cycles in the permutation $F(\Pi)$, defined as follows: given the fixed-point free involution $\sigma_G = \prod_{uv \in E(G)} \left( uv, vu\right)$, let $F(\Pi) = \sigma_G \cdot \prod_{v \in V(G)} \pi_v$.

\tikzset{every picture/.style={line width=0.75pt}} %

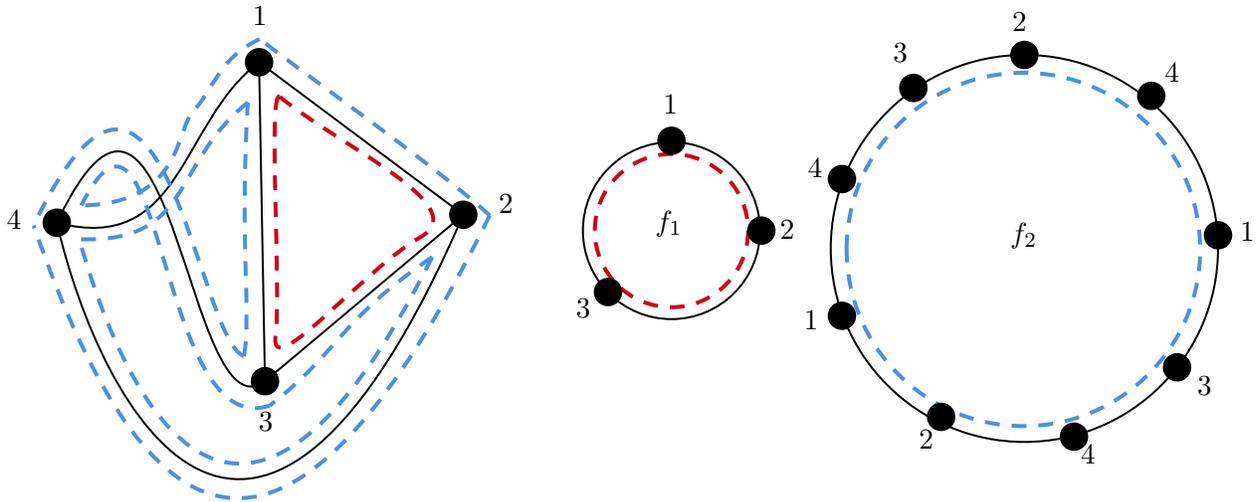
\begin{figure}[ht]
\centering

\tikzset{every picture/.style={line width=0.75pt}} %

\begin{tikzpicture}[x=0.75pt,y=0.75pt,yscale=-1,xscale=1]

\draw    (38.67,128.67) .. controls (97,6) and (103.67,238.67) .. (143.67,208.67) ;
\draw    (140.67,47.67) -- (243.67,124.67) ;
\draw    (140.67,47.67) -- (143.67,208.67) ;
\draw    (143.67,208.67) -- (243.67,124.67) ;
\draw    (38.67,128.67) .. controls (102.33,145.33) and (100.67,77.67) .. (140.67,47.67) ;
\draw    (38.67,128.67) .. controls (61,225) and (139,366) .. (243.67,124.67) ;
\draw [color={rgb, 255:red, 208; green, 2; blue, 27 }  ,draw opacity=1 ][line width=1.5]  [dash pattern={on 5.63pt off 4.5pt}]  (151,65) .. controls (145,63) and (152,162) .. (149,186) .. controls (146,210) and (193.96,151.78) .. (222,136) .. controls (250.04,120.22) and (181,90) .. (151,65) -- cycle ;
\draw [color={rgb, 255:red, 74; green, 144; blue, 226 }  ,draw opacity=1 ][line width=1.5]  [dash pattern={on 5.63pt off 4.5pt}]  (228,146) .. controls (132,338) and (74,222) .. (50,137) ;
\draw [color={rgb, 255:red, 74; green, 144; blue, 226 }  ,draw opacity=1 ][line width=1.5]  [dash pattern={on 5.63pt off 4.5pt}]  (135,68) .. controls (96,99) and (108,139) .. (50,137) ;
\draw [color={rgb, 255:red, 74; green, 144; blue, 226 }  ,draw opacity=1 ][line width=1.5]  [dash pattern={on 5.63pt off 4.5pt}]  (133,196) .. controls (137,176) and (130.67,140.67) .. (135,68) ;
\draw [color={rgb, 255:red, 74; green, 144; blue, 226 }  ,draw opacity=1 ][line width=1.5]  [dash pattern={on 5.63pt off 4.5pt}]  (146,221) .. controls (89,240) and (92,38) .. (51,120) ;
\draw [color={rgb, 255:red, 74; green, 144; blue, 226 }  ,draw opacity=1 ][line width=1.5]  [dash pattern={on 5.63pt off 4.5pt}]  (257,125) .. controls (162,318) and (89,311) .. (27,130) ;
\draw [color={rgb, 255:red, 74; green, 144; blue, 226 }  ,draw opacity=1 ][line width=1.5]  [dash pattern={on 5.63pt off 4.5pt}]  (257,125) .. controls (232,102) and (175,62) .. (141,36) ;
\draw [color={rgb, 255:red, 74; green, 144; blue, 226 }  ,draw opacity=1 ][line width=1.5]  [dash pattern={on 5.63pt off 4.5pt}]  (141,36) .. controls (119.87,45.69) and (117,63) .. (107.9,71.64) .. controls (98.8,80.29) and (102.77,119.13) .. (51,120) ;
\draw [color={rgb, 255:red, 74; green, 144; blue, 226 }  ,draw opacity=1 ][line width=1.5]  [dash pattern={on 5.63pt off 4.5pt}]  (133,196) .. controls (109,187) and (92,-6) .. (27,130) ;
\draw [color={rgb, 255:red, 74; green, 144; blue, 226 }  ,draw opacity=1 ][line width=1.5]  [dash pattern={on 5.63pt off 4.5pt}]  (228,146) .. controls (181,178) and (181,194) .. (146,221) ;
\draw   (304,132.67) .. controls (304,108) and (324,88) .. (348.67,88) .. controls (373.34,88) and (393.33,108) .. (393.33,132.67) .. controls (393.33,157.34) and (373.34,177.33) .. (348.67,177.33) .. controls (324,177.33) and (304,157.34) .. (304,132.67) -- cycle ;
\draw   (429,141.67) .. controls (429,87.73) and (472.73,44) .. (526.67,44) .. controls (580.61,44) and (624.33,87.73) .. (624.33,141.67) .. controls (624.33,195.61) and (580.61,239.33) .. (526.67,239.33) .. controls (472.73,239.33) and (429,195.61) .. (429,141.67) -- cycle ;
\draw  [color={rgb, 255:red, 208; green, 2; blue, 27 }  ,draw opacity=1 ][dash pattern={on 5.63pt off 4.5pt}][line width=1.5]  (310,132.67) .. controls (310,111.31) and (327.31,94) .. (348.67,94) .. controls (370.02,94) and (387.33,111.31) .. (387.33,132.67) .. controls (387.33,154.02) and (370.02,171.33) .. (348.67,171.33) .. controls (327.31,171.33) and (310,154.02) .. (310,132.67) -- cycle ;
\draw  [color={rgb, 255:red, 74; green, 144; blue, 226 }  ,draw opacity=1 ][dash pattern={on 5.63pt off 4.5pt}][line width=1.5]  (437,142.17) .. controls (437,92.92) and (476.92,53) .. (526.17,53) .. controls (575.41,53) and (615.33,92.92) .. (615.33,142.17) .. controls (615.33,191.41) and (575.41,231.33) .. (526.17,231.33) .. controls (476.92,231.33) and (437,191.41) .. (437,142.17) -- cycle ;
\draw  [fill={rgb, 255:red, 0; green, 0; blue, 0 }  ,fill opacity=1 ] (342,87.33) .. controls (342,83.65) and (344.98,80.67) .. (348.67,80.67) .. controls (352.35,80.67) and (355.33,83.65) .. (355.33,87.33) .. controls (355.33,91.02) and (352.35,94) .. (348.67,94) .. controls (344.98,94) and (342,91.02) .. (342,87.33) -- cycle ;
\draw  [fill={rgb, 255:red, 0; green, 0; blue, 0 }  ,fill opacity=1 ] (584,64.67) .. controls (584,60.98) and (586.98,58) .. (590.67,58) .. controls (594.35,58) and (597.33,60.98) .. (597.33,64.67) .. controls (597.33,68.35) and (594.35,71.33) .. (590.67,71.33) .. controls (586.98,71.33) and (584,68.35) .. (584,64.67) -- cycle ;
\draw  [fill={rgb, 255:red, 0; green, 0; blue, 0 }  ,fill opacity=1 ] (310,163.67) .. controls (310,159.98) and (312.98,157) .. (316.67,157) .. controls (320.35,157) and (323.33,159.98) .. (323.33,163.67) .. controls (323.33,167.35) and (320.35,170.33) .. (316.67,170.33) .. controls (312.98,170.33) and (310,167.35) .. (310,163.67) -- cycle ;
\draw  [fill={rgb, 255:red, 0; green, 0; blue, 0 }  ,fill opacity=1 ] (387.33,132.67) .. controls (387.33,128.98) and (390.32,126) .. (394,126) .. controls (397.68,126) and (400.67,128.98) .. (400.67,132.67) .. controls (400.67,136.35) and (397.68,139.33) .. (394,139.33) .. controls (390.32,139.33) and (387.33,136.35) .. (387.33,132.67) -- cycle ;
\draw  [fill={rgb, 255:red, 0; green, 0; blue, 0 }  ,fill opacity=1 ] (545,236.67) .. controls (545,232.98) and (547.98,230) .. (551.67,230) .. controls (555.35,230) and (558.33,232.98) .. (558.33,236.67) .. controls (558.33,240.35) and (555.35,243.33) .. (551.67,243.33) .. controls (547.98,243.33) and (545,240.35) .. (545,236.67) -- cycle ;
\draw  [fill={rgb, 255:red, 0; green, 0; blue, 0 }  ,fill opacity=1 ] (428,106.67) .. controls (428,102.98) and (430.98,100) .. (434.67,100) .. controls (438.35,100) and (441.33,102.98) .. (441.33,106.67) .. controls (441.33,110.35) and (438.35,113.33) .. (434.67,113.33) .. controls (430.98,113.33) and (428,110.35) .. (428,106.67) -- cycle ;
\draw  [fill={rgb, 255:red, 0; green, 0; blue, 0 }  ,fill opacity=1 ] (617.67,135) .. controls (617.67,131.32) and (620.65,128.33) .. (624.33,128.33) .. controls (628.02,128.33) and (631,131.32) .. (631,135) .. controls (631,138.68) and (628.02,141.67) .. (624.33,141.67) .. controls (620.65,141.67) and (617.67,138.68) .. (617.67,135) -- cycle ;
\draw  [fill={rgb, 255:red, 0; green, 0; blue, 0 }  ,fill opacity=1 ] (597,201.67) .. controls (597,197.98) and (599.98,195) .. (603.67,195) .. controls (607.35,195) and (610.33,197.98) .. (610.33,201.67) .. controls (610.33,205.35) and (607.35,208.33) .. (603.67,208.33) .. controls (599.98,208.33) and (597,205.35) .. (597,201.67) -- cycle ;
\draw  [fill={rgb, 255:red, 0; green, 0; blue, 0 }  ,fill opacity=1 ] (478,226.67) .. controls (478,222.98) and (480.98,220) .. (484.67,220) .. controls (488.35,220) and (491.33,222.98) .. (491.33,226.67) .. controls (491.33,230.35) and (488.35,233.33) .. (484.67,233.33) .. controls (480.98,233.33) and (478,230.35) .. (478,226.67) -- cycle ;
\draw  [fill={rgb, 255:red, 0; green, 0; blue, 0 }  ,fill opacity=1 ] (428,175.67) .. controls (428,171.98) and (430.98,169) .. (434.67,169) .. controls (438.35,169) and (441.33,171.98) .. (441.33,175.67) .. controls (441.33,179.35) and (438.35,182.33) .. (434.67,182.33) .. controls (430.98,182.33) and (428,179.35) .. (428,175.67) -- cycle ;
\draw  [fill={rgb, 255:red, 0; green, 0; blue, 0 }  ,fill opacity=1 ] (464,60.67) .. controls (464,56.98) and (466.98,54) .. (470.67,54) .. controls (474.35,54) and (477.33,56.98) .. (477.33,60.67) .. controls (477.33,64.35) and (474.35,67.33) .. (470.67,67.33) .. controls (466.98,67.33) and (464,64.35) .. (464,60.67) -- cycle ;
\draw  [fill={rgb, 255:red, 0; green, 0; blue, 0 }  ,fill opacity=1 ] (520,44) .. controls (520,40.32) and (522.98,37.33) .. (526.67,37.33) .. controls (530.35,37.33) and (533.33,40.32) .. (533.33,44) .. controls (533.33,47.68) and (530.35,50.67) .. (526.67,50.67) .. controls (522.98,50.67) and (520,47.68) .. (520,44) -- cycle ;
\draw  [fill={rgb, 255:red, 0; green, 0; blue, 0 }  ,fill opacity=1 ] (32,128.67) .. controls (32,124.98) and (34.98,122) .. (38.67,122) .. controls (42.35,122) and (45.33,124.98) .. (45.33,128.67) .. controls (45.33,132.35) and (42.35,135.33) .. (38.67,135.33) .. controls (34.98,135.33) and (32,132.35) .. (32,128.67) -- cycle ;
\draw  [fill={rgb, 255:red, 0; green, 0; blue, 0 }  ,fill opacity=1 ] (137,208.67) .. controls (137,204.98) and (139.98,202) .. (143.67,202) .. controls (147.35,202) and (150.33,204.98) .. (150.33,208.67) .. controls (150.33,212.35) and (147.35,215.33) .. (143.67,215.33) .. controls (139.98,215.33) and (137,212.35) .. (137,208.67) -- cycle ;
\draw  [fill={rgb, 255:red, 0; green, 0; blue, 0 }  ,fill opacity=1 ] (237,124.67) .. controls (237,120.98) and (239.98,118) .. (243.67,118) .. controls (247.35,118) and (250.33,120.98) .. (250.33,124.67) .. controls (250.33,128.35) and (247.35,131.33) .. (243.67,131.33) .. controls (239.98,131.33) and (237,128.35) .. (237,124.67) -- cycle ;
\draw  [fill={rgb, 255:red, 0; green, 0; blue, 0 }  ,fill opacity=1 ] (134,47.67) .. controls (134,43.98) and (136.98,41) .. (140.67,41) .. controls (144.35,41) and (147.33,43.98) .. (147.33,47.67) .. controls (147.33,51.35) and (144.35,54.33) .. (140.67,54.33) .. controls (136.98,54.33) and (134,51.35) .. (134,47.67) -- cycle ;

\draw (136,18) node [anchor=north west][inner sep=0.75pt]   [align=left] {1};
\draw (260,113) node [anchor=north west][inner sep=0.75pt]   [align=left] {2};
\draw (139,223) node [anchor=north west][inner sep=0.75pt]   [align=left] {3};
\draw (12,121) node [anchor=north west][inner sep=0.75pt]   [align=left] {4};
\draw (339,121) node [anchor=north west][inner sep=0.75pt]   [align=left] {$\displaystyle f_{1}$};
\draw (518,127) node [anchor=north west][inner sep=0.75pt]   [align=left] {$\displaystyle f_{2}$};
\draw (414,172) node [anchor=north west][inner sep=0.75pt]   [align=left] {1};
\draw (634,127) node [anchor=north west][inner sep=0.75pt]   [align=left] {1};
\draw (343,63) node [anchor=north west][inner sep=0.75pt]   [align=left] {1};
\draw (472,232) node [anchor=north west][inner sep=0.75pt]   [align=left] {2};
\draw (519,21) node [anchor=north west][inner sep=0.75pt]   [align=left] {2};
\draw (402,126) node [anchor=north west][inner sep=0.75pt]   [align=left] {2};
\draw (612.33,204.67) node [anchor=north west][inner sep=0.75pt]   [align=left] {3};
\draw (459,37) node [anchor=north west][inner sep=0.75pt]   [align=left] {3};
\draw (299,166) node [anchor=north west][inner sep=0.75pt]   [align=left] {3};
\draw (416,95) node [anchor=north west][inner sep=0.75pt]   [align=left] {4};
\draw (553.67,239.67) node [anchor=north west][inner sep=0.75pt]   [align=left] {4};
\draw (596,47) node [anchor=north west][inner sep=0.75pt]   [align=left] {4};

\end{tikzpicture}

\vspace*{-20mm}
 \caption{An embedding of $K_4$ with two faces, represented in two different ways: by using local rotations and by exposing the facial walks.}
  \label{k4}  
\end{figure}

Now suppose we have a fixed embedding of some graph $G$, and we want to add a new vertex to this graph. The new vertex $v$ is connected to some vertices of $G$. If $u$ is a neighbor of $v$, we fix one of the appearances of $u$ on the facial walks of $G$. This is where the edge $uv$ will emanate from $u$ in the local clockwise order around $u$. Then, choosing a random local rotation at $v$ gives an embedding of $G' = G + v$. This process can be expressed formally using the language of permutations.  %

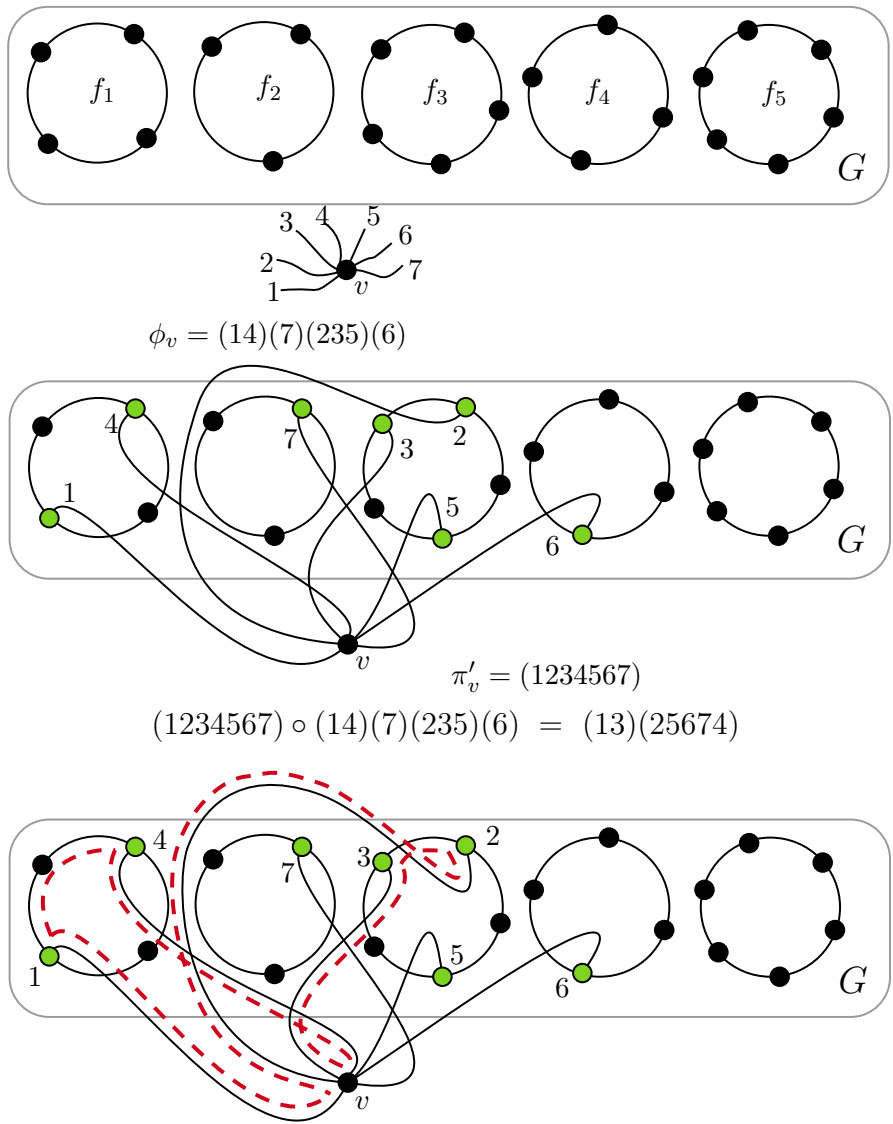
\begin{figure}
\centering

\tikzset{every picture/.style={line width=0.75pt}} %

\begin{tikzpicture}[x=0.75pt,y=0.75pt,yscale=-0.7,xscale=0.7]

\draw   (30,70.17) .. controls (30,42.46) and (52.46,20) .. (80.17,20) .. controls (107.87,20) and (130.33,42.46) .. (130.33,70.17) .. controls (130.33,97.87) and (107.87,120.33) .. (80.17,120.33) .. controls (52.46,120.33) and (30,97.87) .. (30,70.17) -- cycle ;
\draw   (150,69.17) .. controls (150,41.46) and (172.46,19) .. (200.17,19) .. controls (227.87,19) and (250.33,41.46) .. (250.33,69.17) .. controls (250.33,96.87) and (227.87,119.33) .. (200.17,119.33) .. controls (172.46,119.33) and (150,96.87) .. (150,69.17) -- cycle ;
\draw   (271,71.17) .. controls (271,43.46) and (293.46,21) .. (321.17,21) .. controls (348.87,21) and (371.33,43.46) .. (371.33,71.17) .. controls (371.33,98.87) and (348.87,121.33) .. (321.17,121.33) .. controls (293.46,121.33) and (271,98.87) .. (271,71.17) -- cycle ;
\draw   (391,71.17) .. controls (391,43.46) and (413.46,21) .. (441.17,21) .. controls (468.87,21) and (491.33,43.46) .. (491.33,71.17) .. controls (491.33,98.87) and (468.87,121.33) .. (441.17,121.33) .. controls (413.46,121.33) and (391,98.87) .. (391,71.17) -- cycle ;
\draw  [fill={rgb, 255:red, 0; green, 0; blue, 0 }  ,fill opacity=1 ] (253,197.67) .. controls (253,193.98) and (255.98,191) .. (259.67,191) .. controls (263.35,191) and (266.33,193.98) .. (266.33,197.67) .. controls (266.33,201.35) and (263.35,204.33) .. (259.67,204.33) .. controls (255.98,204.33) and (253,201.35) .. (253,197.67) -- cycle ;
\draw  [fill={rgb, 255:red, 0; green, 0; blue, 0 }  ,fill opacity=1 ] (100,27.67) .. controls (100,23.98) and (102.98,21) .. (106.67,21) .. controls (110.35,21) and (113.33,23.98) .. (113.33,27.67) .. controls (113.33,31.35) and (110.35,34.33) .. (106.67,34.33) .. controls (102.98,34.33) and (100,31.35) .. (100,27.67) -- cycle ;
\draw  [fill={rgb, 255:red, 0; green, 0; blue, 0 }  ,fill opacity=1 ] (109,102.67) .. controls (109,98.98) and (111.98,96) .. (115.67,96) .. controls (119.35,96) and (122.33,98.98) .. (122.33,102.67) .. controls (122.33,106.35) and (119.35,109.33) .. (115.67,109.33) .. controls (111.98,109.33) and (109,106.35) .. (109,102.67) -- cycle ;
\draw  [fill={rgb, 255:red, 0; green, 0; blue, 0 }  ,fill opacity=1 ] (33,40.67) .. controls (33,36.98) and (35.98,34) .. (39.67,34) .. controls (43.35,34) and (46.33,36.98) .. (46.33,40.67) .. controls (46.33,44.35) and (43.35,47.33) .. (39.67,47.33) .. controls (35.98,47.33) and (33,44.35) .. (33,40.67) -- cycle ;
\draw  [fill={rgb, 255:red, 0; green, 0; blue, 0 }  ,fill opacity=1 ] (220,27.67) .. controls (220,23.98) and (222.98,21) .. (226.67,21) .. controls (230.35,21) and (233.33,23.98) .. (233.33,27.67) .. controls (233.33,31.35) and (230.35,34.33) .. (226.67,34.33) .. controls (222.98,34.33) and (220,31.35) .. (220,27.67) -- cycle ;
\draw  [fill={rgb, 255:red, 0; green, 0; blue, 0 }  ,fill opacity=1 ] (156,36.67) .. controls (156,32.98) and (158.98,30) .. (162.67,30) .. controls (166.35,30) and (169.33,32.98) .. (169.33,36.67) .. controls (169.33,40.35) and (166.35,43.33) .. (162.67,43.33) .. controls (158.98,43.33) and (156,40.35) .. (156,36.67) -- cycle ;
\draw  [fill={rgb, 255:red, 0; green, 0; blue, 0 }  ,fill opacity=1 ] (200.17,119.33) .. controls (200.17,115.65) and (203.15,112.67) .. (206.83,112.67) .. controls (210.52,112.67) and (213.5,115.65) .. (213.5,119.33) .. controls (213.5,123.02) and (210.52,126) .. (206.83,126) .. controls (203.15,126) and (200.17,123.02) .. (200.17,119.33) -- cycle ;
\draw  [fill={rgb, 255:red, 0; green, 0; blue, 0 }  ,fill opacity=1 ] (363,82.67) .. controls (363,78.98) and (365.98,76) .. (369.67,76) .. controls (373.35,76) and (376.33,78.98) .. (376.33,82.67) .. controls (376.33,86.35) and (373.35,89.33) .. (369.67,89.33) .. controls (365.98,89.33) and (363,86.35) .. (363,82.67) -- cycle ;
\draw  [fill={rgb, 255:red, 0; green, 0; blue, 0 }  ,fill opacity=1 ] (338,26.67) .. controls (338,22.98) and (340.98,20) .. (344.67,20) .. controls (348.35,20) and (351.33,22.98) .. (351.33,26.67) .. controls (351.33,30.35) and (348.35,33.33) .. (344.67,33.33) .. controls (340.98,33.33) and (338,30.35) .. (338,26.67) -- cycle ;
\draw  [fill={rgb, 255:red, 0; green, 0; blue, 0 }  ,fill opacity=1 ] (278,38.67) .. controls (278,34.98) and (280.98,32) .. (284.67,32) .. controls (288.35,32) and (291.33,34.98) .. (291.33,38.67) .. controls (291.33,42.35) and (288.35,45.33) .. (284.67,45.33) .. controls (280.98,45.33) and (278,42.35) .. (278,38.67) -- cycle ;
\draw  [fill={rgb, 255:red, 0; green, 0; blue, 0 }  ,fill opacity=1 ] (272,99.67) .. controls (272,95.98) and (274.98,93) .. (278.67,93) .. controls (282.35,93) and (285.33,95.98) .. (285.33,99.67) .. controls (285.33,103.35) and (282.35,106.33) .. (278.67,106.33) .. controls (274.98,106.33) and (272,103.35) .. (272,99.67) -- cycle ;
\draw  [fill={rgb, 255:red, 0; green, 0; blue, 0 }  ,fill opacity=1 ] (321.17,121.33) .. controls (321.17,117.65) and (324.15,114.67) .. (327.83,114.67) .. controls (331.52,114.67) and (334.5,117.65) .. (334.5,121.33) .. controls (334.5,125.02) and (331.52,128) .. (327.83,128) .. controls (324.15,128) and (321.17,125.02) .. (321.17,121.33) -- cycle ;
\draw  [fill={rgb, 255:red, 0; green, 0; blue, 0 }  ,fill opacity=1 ] (441.17,21) .. controls (441.17,17.32) and (444.15,14.33) .. (447.83,14.33) .. controls (451.52,14.33) and (454.5,17.32) .. (454.5,21) .. controls (454.5,24.68) and (451.52,27.67) .. (447.83,27.67) .. controls (444.15,27.67) and (441.17,24.68) .. (441.17,21) -- cycle ;
\draw  [fill={rgb, 255:red, 0; green, 0; blue, 0 }  ,fill opacity=1 ] (387,58.67) .. controls (387,54.98) and (389.98,52) .. (393.67,52) .. controls (397.35,52) and (400.33,54.98) .. (400.33,58.67) .. controls (400.33,62.35) and (397.35,65.33) .. (393.67,65.33) .. controls (389.98,65.33) and (387,62.35) .. (387,58.67) -- cycle ;
\draw  [fill={rgb, 255:red, 0; green, 0; blue, 0 }  ,fill opacity=1 ] (422,118.67) .. controls (422,114.98) and (424.98,112) .. (428.67,112) .. controls (432.35,112) and (435.33,114.98) .. (435.33,118.67) .. controls (435.33,122.35) and (432.35,125.33) .. (428.67,125.33) .. controls (424.98,125.33) and (422,122.35) .. (422,118.67) -- cycle ;
\draw  [fill={rgb, 255:red, 0; green, 0; blue, 0 }  ,fill opacity=1 ] (481,87.67) .. controls (481,83.98) and (483.98,81) .. (487.67,81) .. controls (491.35,81) and (494.33,83.98) .. (494.33,87.67) .. controls (494.33,91.35) and (491.35,94.33) .. (487.67,94.33) .. controls (483.98,94.33) and (481,91.35) .. (481,87.67) -- cycle ;
\draw  [fill={rgb, 255:red, 0; green, 0; blue, 0 }  ,fill opacity=1 ] (38,106.67) .. controls (38,102.98) and (40.98,100) .. (44.67,100) .. controls (48.35,100) and (51.33,102.98) .. (51.33,106.67) .. controls (51.33,110.35) and (48.35,113.33) .. (44.67,113.33) .. controls (40.98,113.33) and (38,110.35) .. (38,106.67) -- cycle ;
\draw  [color={rgb, 255:red, 155; green, 155; blue, 155 }  ,draw opacity=1 ] (16,36.47) .. controls (16,20.74) and (28.74,8) .. (44.47,8) -- (621.87,8) .. controls (637.59,8) and (650.33,20.74) .. (650.33,36.47) -- (650.33,121.87) .. controls (650.33,137.59) and (637.59,150.33) .. (621.87,150.33) -- (44.47,150.33) .. controls (28.74,150.33) and (16,137.59) .. (16,121.87) -- cycle ;
\draw    (223.33,169.33) .. controls (234.33,176.33) and (247.33,201.33) .. (259.67,197.67) ;
\draw    (273,168) -- (259.67,197.67) ;
\draw    (244,164) .. controls (256,172) and (258.33,188.33) .. (253,197.67) ;
\draw    (300.33,194.33) .. controls (282.33,212.33) and (284.33,196.33) .. (259.67,197.67) ;
\draw    (212.33,212.33) .. controls (246.33,209.33) and (231.33,219.33) .. (259.67,197.67) ;
\draw    (209.33,190.33) .. controls (234.33,194.33) and (225.33,208.33) .. (259.67,197.67) ;
\draw    (259.67,197.67) .. controls (288.33,180.33) and (269.33,198.33) .. (292.33,178.33) ;
\draw   (31,340.17) .. controls (31,312.46) and (53.46,290) .. (81.17,290) .. controls (108.87,290) and (131.33,312.46) .. (131.33,340.17) .. controls (131.33,367.87) and (108.87,390.33) .. (81.17,390.33) .. controls (53.46,390.33) and (31,367.87) .. (31,340.17) -- cycle ;
\draw   (151,339.17) .. controls (151,311.46) and (173.46,289) .. (201.17,289) .. controls (228.87,289) and (251.33,311.46) .. (251.33,339.17) .. controls (251.33,366.87) and (228.87,389.33) .. (201.17,389.33) .. controls (173.46,389.33) and (151,366.87) .. (151,339.17) -- cycle ;
\draw   (272,341.17) .. controls (272,313.46) and (294.46,291) .. (322.17,291) .. controls (349.87,291) and (372.33,313.46) .. (372.33,341.17) .. controls (372.33,368.87) and (349.87,391.33) .. (322.17,391.33) .. controls (294.46,391.33) and (272,368.87) .. (272,341.17) -- cycle ;
\draw   (392,341.17) .. controls (392,313.46) and (414.46,291) .. (442.17,291) .. controls (469.87,291) and (492.33,313.46) .. (492.33,341.17) .. controls (492.33,368.87) and (469.87,391.33) .. (442.17,391.33) .. controls (414.46,391.33) and (392,368.87) .. (392,341.17) -- cycle ;
\draw  [fill={rgb, 255:red, 0; green, 0; blue, 0 }  ,fill opacity=1 ] (254,467.67) .. controls (254,463.98) and (256.98,461) .. (260.67,461) .. controls (264.35,461) and (267.33,463.98) .. (267.33,467.67) .. controls (267.33,471.35) and (264.35,474.33) .. (260.67,474.33) .. controls (256.98,474.33) and (254,471.35) .. (254,467.67) -- cycle ;
\draw  [fill={rgb, 255:red, 0; green, 0; blue, 0 }  ,fill opacity=1 ] (110,372.67) .. controls (110,368.98) and (112.98,366) .. (116.67,366) .. controls (120.35,366) and (123.33,368.98) .. (123.33,372.67) .. controls (123.33,376.35) and (120.35,379.33) .. (116.67,379.33) .. controls (112.98,379.33) and (110,376.35) .. (110,372.67) -- cycle ;
\draw  [fill={rgb, 255:red, 0; green, 0; blue, 0 }  ,fill opacity=1 ] (34,310.67) .. controls (34,306.98) and (36.98,304) .. (40.67,304) .. controls (44.35,304) and (47.33,306.98) .. (47.33,310.67) .. controls (47.33,314.35) and (44.35,317.33) .. (40.67,317.33) .. controls (36.98,317.33) and (34,314.35) .. (34,310.67) -- cycle ;
\draw  [fill={rgb, 255:red, 0; green, 0; blue, 0 }  ,fill opacity=1 ] (157,306.67) .. controls (157,302.98) and (159.98,300) .. (163.67,300) .. controls (167.35,300) and (170.33,302.98) .. (170.33,306.67) .. controls (170.33,310.35) and (167.35,313.33) .. (163.67,313.33) .. controls (159.98,313.33) and (157,310.35) .. (157,306.67) -- cycle ;
\draw  [fill={rgb, 255:red, 0; green, 0; blue, 0 }  ,fill opacity=1 ] (201.17,389.33) .. controls (201.17,385.65) and (204.15,382.67) .. (207.83,382.67) .. controls (211.52,382.67) and (214.5,385.65) .. (214.5,389.33) .. controls (214.5,393.02) and (211.52,396) .. (207.83,396) .. controls (204.15,396) and (201.17,393.02) .. (201.17,389.33) -- cycle ;
\draw  [fill={rgb, 255:red, 0; green, 0; blue, 0 }  ,fill opacity=1 ] (364,352.67) .. controls (364,348.98) and (366.98,346) .. (370.67,346) .. controls (374.35,346) and (377.33,348.98) .. (377.33,352.67) .. controls (377.33,356.35) and (374.35,359.33) .. (370.67,359.33) .. controls (366.98,359.33) and (364,356.35) .. (364,352.67) -- cycle ;
\draw  [fill={rgb, 255:red, 0; green, 0; blue, 0 }  ,fill opacity=1 ] (273,369.67) .. controls (273,365.98) and (275.98,363) .. (279.67,363) .. controls (283.35,363) and (286.33,365.98) .. (286.33,369.67) .. controls (286.33,373.35) and (283.35,376.33) .. (279.67,376.33) .. controls (275.98,376.33) and (273,373.35) .. (273,369.67) -- cycle ;
\draw  [fill={rgb, 255:red, 0; green, 0; blue, 0 }  ,fill opacity=1 ] (442.17,291) .. controls (442.17,287.32) and (445.15,284.33) .. (448.83,284.33) .. controls (452.52,284.33) and (455.5,287.32) .. (455.5,291) .. controls (455.5,294.68) and (452.52,297.67) .. (448.83,297.67) .. controls (445.15,297.67) and (442.17,294.68) .. (442.17,291) -- cycle ;
\draw  [fill={rgb, 255:red, 0; green, 0; blue, 0 }  ,fill opacity=1 ] (388,328.67) .. controls (388,324.98) and (390.98,322) .. (394.67,322) .. controls (398.35,322) and (401.33,324.98) .. (401.33,328.67) .. controls (401.33,332.35) and (398.35,335.33) .. (394.67,335.33) .. controls (390.98,335.33) and (388,332.35) .. (388,328.67) -- cycle ;
\draw  [fill={rgb, 255:red, 0; green, 0; blue, 0 }  ,fill opacity=1 ] (482,357.67) .. controls (482,353.98) and (484.98,351) .. (488.67,351) .. controls (492.35,351) and (495.33,353.98) .. (495.33,357.67) .. controls (495.33,361.35) and (492.35,364.33) .. (488.67,364.33) .. controls (484.98,364.33) and (482,361.35) .. (482,357.67) -- cycle ;
\draw  [color={rgb, 255:red, 155; green, 155; blue, 155 }  ,draw opacity=1 ] (17,306.47) .. controls (17,290.74) and (29.74,278) .. (45.47,278) -- (622.87,278) .. controls (638.59,278) and (651.33,290.74) .. (651.33,306.47) -- (651.33,391.87) .. controls (651.33,407.59) and (638.59,420.33) .. (622.87,420.33) -- (45.47,420.33) .. controls (29.74,420.33) and (17,407.59) .. (17,391.87) -- cycle ;
\draw    (285.67,308.67) .. controls (327.33,343.33) and (171.33,391.33) .. (260.67,467.67) ;
\draw    (107.67,297.67) .. controls (36.33,348.33) and (314.33,438.33) .. (254,467.67) ;
\draw    (227.67,297.67) .. controls (200.33,326.33) and (395.33,491.33) .. (260.67,467.67) ;
\draw    (45.67,376.67) .. controls (62.33,324.33) and (188.33,537.33) .. (260.67,467.67) ;
\draw    (260.67,467.67) .. controls (304.33,440.33) and (492.33,299.33) .. (429.67,388.67) ;
\draw  [fill={rgb, 255:red, 126; green, 211; blue, 33 }  ,fill opacity=1 ] (101,297.67) .. controls (101,293.98) and (103.98,291) .. (107.67,291) .. controls (111.35,291) and (114.33,293.98) .. (114.33,297.67) .. controls (114.33,301.35) and (111.35,304.33) .. (107.67,304.33) .. controls (103.98,304.33) and (101,301.35) .. (101,297.67) -- cycle ;
\draw  [fill={rgb, 255:red, 126; green, 211; blue, 33 }  ,fill opacity=1 ] (221,297.67) .. controls (221,293.98) and (223.98,291) .. (227.67,291) .. controls (231.35,291) and (234.33,293.98) .. (234.33,297.67) .. controls (234.33,301.35) and (231.35,304.33) .. (227.67,304.33) .. controls (223.98,304.33) and (221,301.35) .. (221,297.67) -- cycle ;
\draw  [fill={rgb, 255:red, 126; green, 211; blue, 33 }  ,fill opacity=1 ] (279,308.67) .. controls (279,304.98) and (281.98,302) .. (285.67,302) .. controls (289.35,302) and (292.33,304.98) .. (292.33,308.67) .. controls (292.33,312.35) and (289.35,315.33) .. (285.67,315.33) .. controls (281.98,315.33) and (279,312.35) .. (279,308.67) -- cycle ;
\draw  [fill={rgb, 255:red, 126; green, 211; blue, 33 }  ,fill opacity=1 ] (423,388.67) .. controls (423,384.98) and (425.98,382) .. (429.67,382) .. controls (433.35,382) and (436.33,384.98) .. (436.33,388.67) .. controls (436.33,392.35) and (433.35,395.33) .. (429.67,395.33) .. controls (425.98,395.33) and (423,392.35) .. (423,388.67) -- cycle ;
\draw  [fill={rgb, 255:red, 126; green, 211; blue, 33 }  ,fill opacity=1 ] (39,376.67) .. controls (39,372.98) and (41.98,370) .. (45.67,370) .. controls (49.35,370) and (52.33,372.98) .. (52.33,376.67) .. controls (52.33,380.35) and (49.35,383.33) .. (45.67,383.33) .. controls (41.98,383.33) and (39,380.35) .. (39,376.67) -- cycle ;
\draw    (260.67,467.67) .. controls (300.67,437.67) and (317.33,297.33) .. (328.83,391.33) ;
\draw   (31,656.17) .. controls (31,628.46) and (53.46,606) .. (81.17,606) .. controls (108.87,606) and (131.33,628.46) .. (131.33,656.17) .. controls (131.33,683.87) and (108.87,706.33) .. (81.17,706.33) .. controls (53.46,706.33) and (31,683.87) .. (31,656.17) -- cycle ;
\draw   (151,655.17) .. controls (151,627.46) and (173.46,605) .. (201.17,605) .. controls (228.87,605) and (251.33,627.46) .. (251.33,655.17) .. controls (251.33,682.87) and (228.87,705.33) .. (201.17,705.33) .. controls (173.46,705.33) and (151,682.87) .. (151,655.17) -- cycle ;
\draw   (272,657.17) .. controls (272,629.46) and (294.46,607) .. (322.17,607) .. controls (349.87,607) and (372.33,629.46) .. (372.33,657.17) .. controls (372.33,684.87) and (349.87,707.33) .. (322.17,707.33) .. controls (294.46,707.33) and (272,684.87) .. (272,657.17) -- cycle ;
\draw   (392,657.17) .. controls (392,629.46) and (414.46,607) .. (442.17,607) .. controls (469.87,607) and (492.33,629.46) .. (492.33,657.17) .. controls (492.33,684.87) and (469.87,707.33) .. (442.17,707.33) .. controls (414.46,707.33) and (392,684.87) .. (392,657.17) -- cycle ;
\draw  [fill={rgb, 255:red, 0; green, 0; blue, 0 }  ,fill opacity=1 ] (110,688.67) .. controls (110,684.98) and (112.98,682) .. (116.67,682) .. controls (120.35,682) and (123.33,684.98) .. (123.33,688.67) .. controls (123.33,692.35) and (120.35,695.33) .. (116.67,695.33) .. controls (112.98,695.33) and (110,692.35) .. (110,688.67) -- cycle ;
\draw  [fill={rgb, 255:red, 0; green, 0; blue, 0 }  ,fill opacity=1 ] (34,626.67) .. controls (34,622.98) and (36.98,620) .. (40.67,620) .. controls (44.35,620) and (47.33,622.98) .. (47.33,626.67) .. controls (47.33,630.35) and (44.35,633.33) .. (40.67,633.33) .. controls (36.98,633.33) and (34,630.35) .. (34,626.67) -- cycle ;
\draw  [fill={rgb, 255:red, 0; green, 0; blue, 0 }  ,fill opacity=1 ] (157,622.67) .. controls (157,618.98) and (159.98,616) .. (163.67,616) .. controls (167.35,616) and (170.33,618.98) .. (170.33,622.67) .. controls (170.33,626.35) and (167.35,629.33) .. (163.67,629.33) .. controls (159.98,629.33) and (157,626.35) .. (157,622.67) -- cycle ;
\draw  [fill={rgb, 255:red, 0; green, 0; blue, 0 }  ,fill opacity=1 ] (201.17,705.33) .. controls (201.17,701.65) and (204.15,698.67) .. (207.83,698.67) .. controls (211.52,698.67) and (214.5,701.65) .. (214.5,705.33) .. controls (214.5,709.02) and (211.52,712) .. (207.83,712) .. controls (204.15,712) and (201.17,709.02) .. (201.17,705.33) -- cycle ;
\draw  [fill={rgb, 255:red, 0; green, 0; blue, 0 }  ,fill opacity=1 ] (364,668.67) .. controls (364,664.98) and (366.98,662) .. (370.67,662) .. controls (374.35,662) and (377.33,664.98) .. (377.33,668.67) .. controls (377.33,672.35) and (374.35,675.33) .. (370.67,675.33) .. controls (366.98,675.33) and (364,672.35) .. (364,668.67) -- cycle ;
\draw  [fill={rgb, 255:red, 0; green, 0; blue, 0 }  ,fill opacity=1 ] (273,685.67) .. controls (273,681.98) and (275.98,679) .. (279.67,679) .. controls (283.35,679) and (286.33,681.98) .. (286.33,685.67) .. controls (286.33,689.35) and (283.35,692.33) .. (279.67,692.33) .. controls (275.98,692.33) and (273,689.35) .. (273,685.67) -- cycle ;
\draw  [fill={rgb, 255:red, 0; green, 0; blue, 0 }  ,fill opacity=1 ] (442.17,607) .. controls (442.17,603.32) and (445.15,600.33) .. (448.83,600.33) .. controls (452.52,600.33) and (455.5,603.32) .. (455.5,607) .. controls (455.5,610.68) and (452.52,613.67) .. (448.83,613.67) .. controls (445.15,613.67) and (442.17,610.68) .. (442.17,607) -- cycle ;
\draw  [fill={rgb, 255:red, 0; green, 0; blue, 0 }  ,fill opacity=1 ] (388,644.67) .. controls (388,640.98) and (390.98,638) .. (394.67,638) .. controls (398.35,638) and (401.33,640.98) .. (401.33,644.67) .. controls (401.33,648.35) and (398.35,651.33) .. (394.67,651.33) .. controls (390.98,651.33) and (388,648.35) .. (388,644.67) -- cycle ;
\draw  [fill={rgb, 255:red, 0; green, 0; blue, 0 }  ,fill opacity=1 ] (482,673.67) .. controls (482,669.98) and (484.98,667) .. (488.67,667) .. controls (492.35,667) and (495.33,669.98) .. (495.33,673.67) .. controls (495.33,677.35) and (492.35,680.33) .. (488.67,680.33) .. controls (484.98,680.33) and (482,677.35) .. (482,673.67) -- cycle ;
\draw  [color={rgb, 255:red, 155; green, 155; blue, 155 }  ,draw opacity=1 ] (17,622.47) .. controls (17,606.74) and (29.74,594) .. (45.47,594) -- (622.87,594) .. controls (638.59,594) and (651.33,606.74) .. (651.33,622.47) -- (651.33,707.87) .. controls (651.33,723.59) and (638.59,736.33) .. (622.87,736.33) -- (45.47,736.33) .. controls (29.74,736.33) and (17,723.59) .. (17,707.87) -- cycle ;
\draw    (285.67,624.67) .. controls (327.33,659.33) and (138,727) .. (260.67,783.67) ;
\draw    (227.67,613.67) .. controls (200.33,642.33) and (395.33,807.33) .. (260.67,783.67) ;
\draw    (45.67,692.67) .. controls (62.33,640.33) and (207,892) .. (260.67,783.67) ;
\draw    (345.67,612.67) .. controls (366,698) and (295.64,594.68) .. (238,575) .. controls (180.36,555.32) and (154.31,588.41) .. (145,615) .. controls (135.69,641.59) and (136.33,776.33) .. (260.67,783.67) ;
\draw    (260.67,783.67) .. controls (304.33,756.33) and (492.33,615.33) .. (429.67,704.67) ;
\draw  [fill={rgb, 255:red, 126; green, 211; blue, 33 }  ,fill opacity=1 ] (221,613.67) .. controls (221,609.98) and (223.98,607) .. (227.67,607) .. controls (231.35,607) and (234.33,609.98) .. (234.33,613.67) .. controls (234.33,617.35) and (231.35,620.33) .. (227.67,620.33) .. controls (223.98,620.33) and (221,617.35) .. (221,613.67) -- cycle ;
\draw  [fill={rgb, 255:red, 126; green, 211; blue, 33 }  ,fill opacity=1 ] (423,704.67) .. controls (423,700.98) and (425.98,698) .. (429.67,698) .. controls (433.35,698) and (436.33,700.98) .. (436.33,704.67) .. controls (436.33,708.35) and (433.35,711.33) .. (429.67,711.33) .. controls (425.98,711.33) and (423,708.35) .. (423,704.67) -- cycle ;
\draw    (260.67,783.67) .. controls (300.67,753.67) and (317.33,613.33) .. (328.83,707.33) ;
\draw [color={rgb, 255:red, 208; green, 2; blue, 27 }  ,draw opacity=1 ][line width=1.5]  [dash pattern={on 5.63pt off 4.5pt}]  (93,616) .. controls (59,690) and (303,762) .. (257,772) ;
\draw [color={rgb, 255:red, 208; green, 2; blue, 27 }  ,draw opacity=1 ][line width=1.5]  [dash pattern={on 5.63pt off 4.5pt}]  (257,772) .. controls (161,731) and (325,666) .. (299,626) ;
\draw [color={rgb, 255:red, 208; green, 2; blue, 27 }  ,draw opacity=1 ][line width=1.5]  [dash pattern={on 5.63pt off 4.5pt}]  (299,626) .. controls (305,616) and (312,615) .. (333,617) ;
\draw [color={rgb, 255:red, 208; green, 2; blue, 27 }  ,draw opacity=1 ][line width=1.5]  [dash pattern={on 5.63pt off 4.5pt}]  (333,617) .. controls (367,676) and (277,580) .. (241,568) .. controls (205,556) and (205,561) .. (187,563) .. controls (169,565) and (143.35,588.18) .. (136,621) .. controls (128.65,653.82) and (143,726) .. (174,756) .. controls (205,786) and (236,783) .. (248,792) ;
\draw  [fill={rgb, 255:red, 126; green, 211; blue, 33 }  ,fill opacity=1 ] (322.17,707.33) .. controls (322.17,703.65) and (325.15,700.67) .. (328.83,700.67) .. controls (332.52,700.67) and (335.5,703.65) .. (335.5,707.33) .. controls (335.5,711.02) and (332.52,714) .. (328.83,714) .. controls (325.15,714) and (322.17,711.02) .. (322.17,707.33) -- cycle ;
\draw [color={rgb, 255:red, 208; green, 2; blue, 27 }  ,draw opacity=1 ][line width=1.5]  [dash pattern={on 5.63pt off 4.5pt}]  (47,678) .. controls (23,624) and (77,614) .. (93,616) ;
\draw    (107.67,613.67) .. controls (72.51,638.64) and (124.16,680.47) .. (179,710) .. controls (233.84,739.53) and (295.08,758.48) .. (254,783.67) ;
\draw   (514,71.17) .. controls (514,43.46) and (536.46,21) .. (564.17,21) .. controls (591.87,21) and (614.33,43.46) .. (614.33,71.17) .. controls (614.33,98.87) and (591.87,121.33) .. (564.17,121.33) .. controls (536.46,121.33) and (514,98.87) .. (514,71.17) -- cycle ;
\draw  [fill={rgb, 255:red, 0; green, 0; blue, 0 }  ,fill opacity=1 ] (542.17,25) .. controls (542.17,21.32) and (545.15,18.33) .. (548.83,18.33) .. controls (552.52,18.33) and (555.5,21.32) .. (555.5,25) .. controls (555.5,28.68) and (552.52,31.67) .. (548.83,31.67) .. controls (545.15,31.67) and (542.17,28.68) .. (542.17,25) -- cycle ;
\draw  [fill={rgb, 255:red, 0; green, 0; blue, 0 }  ,fill opacity=1 ] (510,58.67) .. controls (510,54.98) and (512.98,52) .. (516.67,52) .. controls (520.35,52) and (523.33,54.98) .. (523.33,58.67) .. controls (523.33,62.35) and (520.35,65.33) .. (516.67,65.33) .. controls (512.98,65.33) and (510,62.35) .. (510,58.67) -- cycle ;
\draw  [fill={rgb, 255:red, 0; green, 0; blue, 0 }  ,fill opacity=1 ] (520,103.67) .. controls (520,99.98) and (522.98,97) .. (526.67,97) .. controls (530.35,97) and (533.33,99.98) .. (533.33,103.67) .. controls (533.33,107.35) and (530.35,110.33) .. (526.67,110.33) .. controls (522.98,110.33) and (520,107.35) .. (520,103.67) -- cycle ;
\draw  [fill={rgb, 255:red, 0; green, 0; blue, 0 }  ,fill opacity=1 ] (604,87.67) .. controls (604,83.98) and (606.98,81) .. (610.67,81) .. controls (614.35,81) and (617.33,83.98) .. (617.33,87.67) .. controls (617.33,91.35) and (614.35,94.33) .. (610.67,94.33) .. controls (606.98,94.33) and (604,91.35) .. (604,87.67) -- cycle ;
\draw  [fill={rgb, 255:red, 0; green, 0; blue, 0 }  ,fill opacity=1 ] (595.17,39) .. controls (595.17,35.32) and (598.15,32.33) .. (601.83,32.33) .. controls (605.52,32.33) and (608.5,35.32) .. (608.5,39) .. controls (608.5,42.68) and (605.52,45.67) .. (601.83,45.67) .. controls (598.15,45.67) and (595.17,42.68) .. (595.17,39) -- cycle ;
\draw  [fill={rgb, 255:red, 0; green, 0; blue, 0 }  ,fill opacity=1 ] (564.17,121.33) .. controls (564.17,117.65) and (567.15,114.67) .. (570.83,114.67) .. controls (574.52,114.67) and (577.5,117.65) .. (577.5,121.33) .. controls (577.5,125.02) and (574.52,128) .. (570.83,128) .. controls (567.15,128) and (564.17,125.02) .. (564.17,121.33) -- cycle ;
\draw   (514,339.17) .. controls (514,311.46) and (536.46,289) .. (564.17,289) .. controls (591.87,289) and (614.33,311.46) .. (614.33,339.17) .. controls (614.33,366.87) and (591.87,389.33) .. (564.17,389.33) .. controls (536.46,389.33) and (514,366.87) .. (514,339.17) -- cycle ;
\draw  [fill={rgb, 255:red, 0; green, 0; blue, 0 }  ,fill opacity=1 ] (542.17,293) .. controls (542.17,289.32) and (545.15,286.33) .. (548.83,286.33) .. controls (552.52,286.33) and (555.5,289.32) .. (555.5,293) .. controls (555.5,296.68) and (552.52,299.67) .. (548.83,299.67) .. controls (545.15,299.67) and (542.17,296.68) .. (542.17,293) -- cycle ;
\draw  [fill={rgb, 255:red, 0; green, 0; blue, 0 }  ,fill opacity=1 ] (510,326.67) .. controls (510,322.98) and (512.98,320) .. (516.67,320) .. controls (520.35,320) and (523.33,322.98) .. (523.33,326.67) .. controls (523.33,330.35) and (520.35,333.33) .. (516.67,333.33) .. controls (512.98,333.33) and (510,330.35) .. (510,326.67) -- cycle ;
\draw  [fill={rgb, 255:red, 0; green, 0; blue, 0 }  ,fill opacity=1 ] (520,371.67) .. controls (520,367.98) and (522.98,365) .. (526.67,365) .. controls (530.35,365) and (533.33,367.98) .. (533.33,371.67) .. controls (533.33,375.35) and (530.35,378.33) .. (526.67,378.33) .. controls (522.98,378.33) and (520,375.35) .. (520,371.67) -- cycle ;
\draw  [fill={rgb, 255:red, 0; green, 0; blue, 0 }  ,fill opacity=1 ] (604,355.67) .. controls (604,351.98) and (606.98,349) .. (610.67,349) .. controls (614.35,349) and (617.33,351.98) .. (617.33,355.67) .. controls (617.33,359.35) and (614.35,362.33) .. (610.67,362.33) .. controls (606.98,362.33) and (604,359.35) .. (604,355.67) -- cycle ;
\draw  [fill={rgb, 255:red, 0; green, 0; blue, 0 }  ,fill opacity=1 ] (595.17,307) .. controls (595.17,303.32) and (598.15,300.33) .. (601.83,300.33) .. controls (605.52,300.33) and (608.5,303.32) .. (608.5,307) .. controls (608.5,310.68) and (605.52,313.67) .. (601.83,313.67) .. controls (598.15,313.67) and (595.17,310.68) .. (595.17,307) -- cycle ;
\draw  [fill={rgb, 255:red, 0; green, 0; blue, 0 }  ,fill opacity=1 ] (564.17,389.33) .. controls (564.17,385.65) and (567.15,382.67) .. (570.83,382.67) .. controls (574.52,382.67) and (577.5,385.65) .. (577.5,389.33) .. controls (577.5,393.02) and (574.52,396) .. (570.83,396) .. controls (567.15,396) and (564.17,393.02) .. (564.17,389.33) -- cycle ;
\draw   (515,657.17) .. controls (515,629.46) and (537.46,607) .. (565.17,607) .. controls (592.87,607) and (615.33,629.46) .. (615.33,657.17) .. controls (615.33,684.87) and (592.87,707.33) .. (565.17,707.33) .. controls (537.46,707.33) and (515,684.87) .. (515,657.17) -- cycle ;
\draw  [fill={rgb, 255:red, 0; green, 0; blue, 0 }  ,fill opacity=1 ] (543.17,611) .. controls (543.17,607.32) and (546.15,604.33) .. (549.83,604.33) .. controls (553.52,604.33) and (556.5,607.32) .. (556.5,611) .. controls (556.5,614.68) and (553.52,617.67) .. (549.83,617.67) .. controls (546.15,617.67) and (543.17,614.68) .. (543.17,611) -- cycle ;
\draw  [fill={rgb, 255:red, 0; green, 0; blue, 0 }  ,fill opacity=1 ] (511,644.67) .. controls (511,640.98) and (513.98,638) .. (517.67,638) .. controls (521.35,638) and (524.33,640.98) .. (524.33,644.67) .. controls (524.33,648.35) and (521.35,651.33) .. (517.67,651.33) .. controls (513.98,651.33) and (511,648.35) .. (511,644.67) -- cycle ;
\draw  [fill={rgb, 255:red, 0; green, 0; blue, 0 }  ,fill opacity=1 ] (521,689.67) .. controls (521,685.98) and (523.98,683) .. (527.67,683) .. controls (531.35,683) and (534.33,685.98) .. (534.33,689.67) .. controls (534.33,693.35) and (531.35,696.33) .. (527.67,696.33) .. controls (523.98,696.33) and (521,693.35) .. (521,689.67) -- cycle ;
\draw  [fill={rgb, 255:red, 0; green, 0; blue, 0 }  ,fill opacity=1 ] (605,673.67) .. controls (605,669.98) and (607.98,667) .. (611.67,667) .. controls (615.35,667) and (618.33,669.98) .. (618.33,673.67) .. controls (618.33,677.35) and (615.35,680.33) .. (611.67,680.33) .. controls (607.98,680.33) and (605,677.35) .. (605,673.67) -- cycle ;
\draw  [fill={rgb, 255:red, 0; green, 0; blue, 0 }  ,fill opacity=1 ] (596.17,625) .. controls (596.17,621.32) and (599.15,618.33) .. (602.83,618.33) .. controls (606.52,618.33) and (609.5,621.32) .. (609.5,625) .. controls (609.5,628.68) and (606.52,631.67) .. (602.83,631.67) .. controls (599.15,631.67) and (596.17,628.68) .. (596.17,625) -- cycle ;
\draw  [fill={rgb, 255:red, 0; green, 0; blue, 0 }  ,fill opacity=1 ] (565.17,707.33) .. controls (565.17,703.65) and (568.15,700.67) .. (571.83,700.67) .. controls (575.52,700.67) and (578.5,703.65) .. (578.5,707.33) .. controls (578.5,711.02) and (575.52,714) .. (571.83,714) .. controls (568.15,714) and (565.17,711.02) .. (565.17,707.33) -- cycle ;
\draw  [fill={rgb, 255:red, 126; green, 211; blue, 33 }  ,fill opacity=1 ] (322.17,391.33) .. controls (322.17,387.65) and (325.15,384.67) .. (328.83,384.67) .. controls (332.52,384.67) and (335.5,387.65) .. (335.5,391.33) .. controls (335.5,395.02) and (332.52,398) .. (328.83,398) .. controls (325.15,398) and (322.17,395.02) .. (322.17,391.33) -- cycle ;
\draw  [fill={rgb, 255:red, 126; green, 211; blue, 33 }  ,fill opacity=1 ] (101,613.67) .. controls (101,609.98) and (103.98,607) .. (107.67,607) .. controls (111.35,607) and (114.33,609.98) .. (114.33,613.67) .. controls (114.33,617.35) and (111.35,620.33) .. (107.67,620.33) .. controls (103.98,620.33) and (101,617.35) .. (101,613.67) -- cycle ;
\draw [color={rgb, 255:red, 208; green, 2; blue, 27 }  ,draw opacity=1 ][line width=1.5]  [dash pattern={on 5.63pt off 4.5pt}]  (47,678) .. controls (75,646) and (185,845) .. (248,792) ;
\draw  [fill={rgb, 255:red, 0; green, 0; blue, 0 }  ,fill opacity=1 ] (254,783.67) .. controls (254,779.98) and (256.98,777) .. (260.67,777) .. controls (264.35,777) and (267.33,779.98) .. (267.33,783.67) .. controls (267.33,787.35) and (264.35,790.33) .. (260.67,790.33) .. controls (256.98,790.33) and (254,787.35) .. (254,783.67) -- cycle ;
\draw  [fill={rgb, 255:red, 126; green, 211; blue, 33 }  ,fill opacity=1 ] (279,624.67) .. controls (279,620.98) and (281.98,618) .. (285.67,618) .. controls (289.35,618) and (292.33,620.98) .. (292.33,624.67) .. controls (292.33,628.35) and (289.35,631.33) .. (285.67,631.33) .. controls (281.98,631.33) and (279,628.35) .. (279,624.67) -- cycle ;
\draw  [fill={rgb, 255:red, 126; green, 211; blue, 33 }  ,fill opacity=1 ] (339,612.67) .. controls (339,608.98) and (341.98,606) .. (345.67,606) .. controls (349.35,606) and (352.33,608.98) .. (352.33,612.67) .. controls (352.33,616.35) and (349.35,619.33) .. (345.67,619.33) .. controls (341.98,619.33) and (339,616.35) .. (339,612.67) -- cycle ;
\draw  [fill={rgb, 255:red, 126; green, 211; blue, 33 }  ,fill opacity=1 ] (39,692.67) .. controls (39,688.98) and (41.98,686) .. (45.67,686) .. controls (49.35,686) and (52.33,688.98) .. (52.33,692.67) .. controls (52.33,696.35) and (49.35,699.33) .. (45.67,699.33) .. controls (41.98,699.33) and (39,696.35) .. (39,692.67) -- cycle ;
\draw    (345.67,296.67) .. controls (320.33,345.33) and (182,212.33) .. (153,294) .. controls (124,375.67) and (136.33,460.33) .. (260.67,467.67) ;
\draw  [fill={rgb, 255:red, 126; green, 211; blue, 33 }  ,fill opacity=1 ] (339,296.67) .. controls (339,292.98) and (341.98,290) .. (345.67,290) .. controls (349.35,290) and (352.33,292.98) .. (352.33,296.67) .. controls (352.33,300.35) and (349.35,303.33) .. (345.67,303.33) .. controls (341.98,303.33) and (339,300.35) .. (339,296.67) -- cycle ;

\draw (72,57) node [anchor=north west][inner sep=0.75pt]   [align=left] {$\displaystyle f_{1}$};
\draw (429,58) node [anchor=north west][inner sep=0.75pt]   [align=left] {$\displaystyle f_{4}$};
\draw (312,58) node [anchor=north west][inner sep=0.75pt]   [align=left] {$\displaystyle f_{3}$};
\draw (192,55) node [anchor=north west][inner sep=0.75pt]   [align=left] {$\displaystyle f_{2}$};
\draw (263,204) node [anchor=north west][inner sep=0.75pt]   [align=left] {$\displaystyle v$};
\draw (611,112) node [anchor=north west][inner sep=0.75pt]  [font=\Large] [align=left] {$\displaystyle G$};
\draw (556,59) node [anchor=north west][inner sep=0.75pt]   [align=left] {$\displaystyle f_{5}$};
\draw (264,474) node [anchor=north west][inner sep=0.75pt]   [align=left] {$\displaystyle v$};
\draw (200,205) node [anchor=north west][inner sep=0.75pt]   [align=left] {1};
\draw (195,184) node [anchor=north west][inner sep=0.75pt]   [align=left] {2};
\draw (209,154) node [anchor=north west][inner sep=0.75pt]   [align=left] {3};
\draw (235,150) node [anchor=north west][inner sep=0.75pt]   [align=left] {4};
\draw (272,150) node [anchor=north west][inner sep=0.75pt]   [align=left] {5};
\draw (295,164) node [anchor=north west][inner sep=0.75pt]   [align=left] {6};
\draw (302,189) node [anchor=north west][inner sep=0.75pt]   [align=left] {7};
\draw (53,350) node [anchor=north west][inner sep=0.75pt]   [align=left] {1};
\draw (334,309) node [anchor=north west][inner sep=0.75pt]   [align=left] {2};
\draw (296,318) node [anchor=north west][inner sep=0.75pt]   [align=left] {3};
\draw (83,300) node [anchor=north west][inner sep=0.75pt]   [align=left] {4};
\draw (329,360) node [anchor=north west][inner sep=0.75pt]   [align=left] {5};
\draw (401,388) node [anchor=north west][inner sep=0.75pt]   [align=left] {6};
\draw (212,311) node [anchor=north west][inner sep=0.75pt]   [align=left] {7};
\draw (117,513.33) node [anchor=north west][inner sep=0.75pt]  [font=\large] [align=left] {$\displaystyle ( 1234567) \circ ( 14)( 7)( 235)( 6) \ =\ ( 13)( 25674)$};
\draw (333,477.33) node [anchor=north west][inner sep=0.75pt]   [align=left] {$\displaystyle \pi _{v} '=( 1234567)$};
\draw (115,232.33) node [anchor=north west][inner sep=0.75pt]   [align=left] {$\displaystyle \phi _{v} =( 14)( 7)( 235)( 6)$};
\draw (263.67,791.33) node [anchor=north west][inner sep=0.75pt]   [align=left] {$\displaystyle v$};
\draw (28,699) node [anchor=north west][inner sep=0.75pt]   [align=left] {1};
\draw (358,597) node [anchor=north west][inner sep=0.75pt]   [align=left] {2};
\draw (265,612) node [anchor=north west][inner sep=0.75pt]   [align=left] {3};
\draw (118,600) node [anchor=north west][inner sep=0.75pt]   [align=left] {4};
\draw (332,683) node [anchor=north west][inner sep=0.75pt]   [align=left] {5};
\draw (408,707) node [anchor=north west][inner sep=0.75pt]   [align=left] {6};
\draw (210,624) node [anchor=north west][inner sep=0.75pt]   [align=left] {7};
\draw (611,380) node [anchor=north west][inner sep=0.75pt]  [font=\Large] [align=left] {$\displaystyle G$};
\draw (612,698) node [anchor=north west][inner sep=0.75pt]  [font=\Large] [align=left] {$\displaystyle G$};

\end{tikzpicture}

\vspace*{-10mm}
 \caption{An example of the process of adding a vertex $v$ to an embedding of a graph $G$.} 
  \label{fig:newvert}
\end{figure}

\begin{definition}\label{def:addvx}
Let $G'$ be a graph and $v \in V(G')$. Let $G=G'-v$ and consider an embedding $\Pi=\{ \pi_u \mid u \in V(G)\}$ of $G$. Now, consider all embeddings $\Pi'$ of $G'$ obtained by the following random process:
\begin{enumerate}
\item For each vertex $u$ adjacent to $v$, insert the dart $uv$ into the local rotation $\pi_u$ uniformly at random into one of the gaps between consecutive darts in $\pi_u$ to obtain a cyclic permutation $\pi'_u$.
\item Secondly, fix a uniformly random cyclic permutation $\pi'_v$ of the darts incident to $v$.
\end{enumerate}
We say that the random embedding $\Pi'=\{ \pi'_u \mid u \in V(G')\}$ resulting from this process from $\Pi$ was obtained by \emph{randomly adding $v$ to $\Pi$}. 
\end{definition}

For our purpose, the crucial observation to be made regarding the process outlined in Definition~\ref{def:addvx} is that it can be used to build a random embedding of a graph one vertex at a time.

\begin{lemma}
Given a graph $G'$ and a vertex $v \in V(G')$, if $\Pi$ is a uniformly random embedding of $G = G'-v$, then randomly adding $v$ to $\Pi$ gives a uniformly random embedding $\Pi'$ of $G'$.
\end{lemma}

\begin{proof}
We want to show that, for any fixed embedding of $G'$, say $\Pi'_0$, the probability of obtaining $\Pi'_0=\{ \pi'_u \mid u \in V(G)\}$ in the manner described is $\prod_{u \in V(G')} \frac{1}{(\deg_{G'}(u)-1)!}$. Indeed, this process will give $\Pi'_0$ only if $\Pi=\Pi_0$, where $\Pi_0=\{ \pi_u \mid u \in V(G)\}$ is the unique embedding of $G$ in which each $\pi_u$ is obtained from $\pi'_u$ by removing any dart with head $v$. Of course, $\Pi = \Pi_0$ with probability $\prod_{u \in V(G)} \frac{1}{(\deg_{G}(u)-1)!}$. Then, for each vertex $u$ which is adjacent to $v$ in $G'$, exactly one of the $\deg_G(u)$ possible placements of $uv$ into  $\pi_u$ will yield $\pi'_u$. Noticing that $\deg_{G}(u) = \deg_{G'}(u)-1$ whenever $u$ is adjacent to $v$, and for all other vertices we have $\deg_{G}(w) = \deg_{G'}(w)$, we see that the correct rotation at every vertex besides $v$ is obtained with probability $\prod_{u \in V(G)\setminus \{v\}} \frac{1}{(\deg_{G'}(u)-1)!}$. The result then follows from the fact that we pick the rotation $\pi'_v$ at $v$ with probability $\frac{1}{(\deg_{G'}(v)-1)!}$, and this choice is independent of all previous choices.
\end{proof}

We give an example of this process in Figure~\ref{fig:newvert}. Here we have an embedding $\Pi$ of a graph $G$ with five facial walks. Because the darts in $\pi'_v$ all have the same tail, namely $v$, we can represent it more simply as a permutation of vertices, corresponding to the head of each dart. In particular, we set $\pi'_v = (1 \, 2 \, 3 \, 4 \, 5 \, 6 \, 7)$. Then, for each $i \in \{1,2,\ldots,7\}$, we randomly select one of the appearances of $i$ in the facial walks of $\Pi$ to connect the edge $vi$. In this example,  $v$ is connected only to four faces in $G$, and it breaks all of these faces in the new embedding of $G + v$. The fifth face, however, remains unchanged.  By tracing around the new faces, we can see that adding $v$ gives us two new faces: the tracing for one of the new faces is shown in a red dashed line. %

Let us recall that the facial walks of the embedding $\Pi$ correspond to the cycles of the permutation of all darts, $F(\Pi) = \sigma_G \cdot \prod_{v \in V(G)} \pi_v$, where $\sigma_G$ is the involution that swaps each dart $xy$ incident with $x$ with the reverse dart $yx$ incident with $y$.  
To describe the faces of the embedding $\Pi'$ obtained by (randomly) adding the vertex $v$ to $\Pi$, we introduce the \emph{face permutation of $v$ relative to $\Pi'$}, denoted $\phi_v$, which is the subpermutation of $F(\Pi')$ induced by the darts incident to $v$.

For the example in Figure \ref{fig:newvert}, we can determine $\phi_v$, the face permutation of $v$ relative to $\Pi'$, by reading the anticlockwise cyclic ordering of edges incident to $v$ entering each face of $\Pi$. For instance, the cyclic traversal of the face $f_3$ gives the cycle $(2\, 3\, 5)$ of $\phi_v$. Each face in $\Pi$ gives a (possibly empty) cycle of $\phi_v$. Again using the vertex at the head of each dart incident to $v$ as a unique identifier, we can carry out this process in the given example to obtain the permutation $\phi_v = (1 \, 4)(7)(2 \, 3 \, 5)(6)$. The key observation is now that taking the product of the rotation scheme at $v$ with this permutation gives us the permutation describing the faces in $\Pi'$ which contain $v$.  This is shown in the example:  $(1 \, 2 \, 3 \, 4 \, 5 \, 6 \, 7) \circ (1 \, 4)(7)(2 \, 3 \, 5)(6) = (1 \, 3)(2 \, 5 \, 6 \, 7 \, 4)$, and the two new faces added are described by this permutation. We state this observation more formally as the following lemma whose straightforward proof is left to the reader. 

\begin{lemma} \label{lemma:addcycles}
Suppose $\Pi'$ is the embedding obtained by randomly adding $v$ to $\Pi$. Then the faces of $\Pi'$ which contain $v$ correspond to the cycles of $\pi'_v\phi_v$.
\end{lemma}

In the next theorem we combine the framework developed in this section with the result about multistars
that was proved in Theorem~\ref{thm:asymmulti}; 
we crucially use that the bound in Theorem~\ref{thm:asymmulti} depends 
on $\lambda$ only in a limited way. 
Recall that $\Delta_d$ was defined as $H_{d-1} + \left\lceil \tfrac{d}{2} \right\rceil^{-1}$.

\begin{theorem} \label{thm:newvert}
  Suppose a vertex $v$ of degree $d \ge 1$ is randomly added to an embedding of a graph $G$.  
  Let $h(d) = \Delta_d$ for $1 \le d \le 4$ and $h(d) = \Delta_d + \frac{1}{d+1}$ for $d \ge 5$. 
  Then the expected number of faces containing $v$ is at most $h(d)$. 
\end{theorem}

\begin{proof}
  Let $d = \deg_{G'}(v)$ and arbitrarily assign the darts incident to $v$ integers in the set $\{1,2,\ldots,d\}$. Step 1 of the process described in Definition \ref{def:addvx} then defines a partition $\lambda \vdash d$: we say that $vu$ and $vw$ are equivalent if they are inserted into the same face of $\Pi$.

  We do the estimate separately for each possible outcome of Step~1. 
  When we fix any such outcome, we also fix $\lambda \vdash d$ and a face permutation $\phi_v \in C_\lambda$.
  By Lemma~\ref{lemma:addcycles}
  we need to estimate the expected number of cycles of $\pi'_v \circ \phi_v$, with $\phi_v$ being a fixed permutation in $C_\lambda$
  and $\pi'_v$ a uniformly random unicyclic permutation. By symmetry, this number is the same for every $\phi_v \in C_\lambda$ (with $\lambda$ fixed) and thus it is also the same as the expected number of faces of a multistar $K_\lambda(d)$.
  For $d\ge 5$ we apply Theorem~\ref{thm:asymmulti} which tells us that the expected value is at most $h(d')$ for $5\le d' = d-r(\lambda)$.
  For $d'=0$, vertex $v$ trivially contains only one face.
  For $d'=2$, vertex $v$ creates a dipole with two edges. %
  The same is true for $d'=3$. %
  For $d'=4$, we observe that the only non-trivial partition that does not lead to a dipole is $\lambda=(2,2)$.
We can easily verify that the multistar corresponding to such $\lambda$ has the expected number of faces $\frac{7}{3}$, which is exactly the same as the dipole with four edges.
Thus for $1\le d\le4$ the final bound follows by Corollary~\ref{cor:dipoleE}.
 For $d\ge 5$ we use the above bounds for each version of Step 1 of Definition~\ref{def:addvx}.
 We conclude that the expected number of faces containing $v$ is at most $h(d)$ as $h$ is monotone. %
  This completes the proof. 
\end{proof}

We can use Theorem~\ref{thm:newvert} to find general bounds on the expected number of faces in random embeddings of graphs. Given a graph $G$, let $F$ denote the random variable for the number of faces in a random embedding of $G$. It immediately follows from Theorem~\ref{thm:newvert} that, if $G$ has maximum degree $d\ge 2$, then 
$$\EE[F] < n \, \bigl(\log(d)+\tfrac{5}{3}\bigr).$$ 
This result can be strengthened to allow vertices of degree larger than $d$ as long as the ``maximum average degree'' is at most $d$. The notion of the maximum average degree comes from the theory of sparse graphs (see \cite{NOdM12}). To be more precise, we say that a graph is \emph{$d$-degenerate} if there is an ordering $v_1,\ldots,v_n$ of its vertices such that $d_i \le d$ for each $i\in [n]$, where $d_i$ denotes the number of neighbors $v_j$ of $v_i$ with $j<i$. We call $d_i$ the \emph{back-degree} of the vertex $v_i$. We make a slight change to the definition above:
whenever $d_i=0$ we redefine it to the value $d_i=1$ instead.

\begin{theorem}\label{thm:degbnd}
  Let $G$ be a connected graph of order $n$. Given a linear order of vertices $v_1,\dots,v_n$ with respective back-degrees $d_i$ $(1\le i \le n)$, 
  we have.
\begin{equation}
     \EE [F] \le 1 + \sum_{i=3}^n \log {d_i^*},
     \label{eq:thm8-1}
\end{equation}
where $d_i^*:=d_i$ if $d_i\neq 2$ and $d_i^*:=e$ if $d_i=2$. We also have
\begin{equation}
     \EE [F] \le 1 + \sum_{i=3}^n H_{d_i-1}.
     \label{eq:thm8-2}
\end{equation}
\end{theorem}

\begin{proof}
We are building $G$ starting with $v_1$ and adding vertices $v_2,\dots,v_n$ to the current graph using the given order.
The proof is by induction on $n$.
The base of induction is for connected graphs on at most two vertices that clearly have only one face.

In the induction step we consider $n\geq 3$.
Consider $G':=G\setminus v_n$.
By Theorem~\ref{thm:newvert}, $v_n$ is in average in at most $H_{d_n-1} + \bigl\lceil \tfrac{d_n}{2} \bigr\rceil^{-1} + x_{d_n}$ faces, where $x_1=x_2=x_3=x_4=0$ and $x_{d_n}=\frac{1}{d_n+1}$ for $d_n \ge 5$.
Moreover, since $G$ is connected, $v_n$ is connected to each component of $G'$, and by building any embedding of $G'+v_n$, we have destroyed at least one face from each component.
Let $C_1,\ldots,C_k$ be the components of $G'$. 
We conclude the proof by the following computation: 
\[
   \EE[F] \le \sum_{i=1}^{k} \left(1+\sum_{v_j\in V(C_i)} H_{d_j-1} \right) - k + H_{d_n-1} + \left\lceil \tfrac{d_n}{2} \right\rceil^{-1}+x_{d_n}
   \le 1 + \sum_{i=3}^n  H_{d_i-1}
\]
         and
\[
   \EE[F] \le \sum_{i=1}^{k} \left(1+\sum_{v_j\in V(C_i)}\log d_j^* \right) -k + H_{d_n-1} + \left\lceil \tfrac{d_n}{2} \right\rceil^{-1}+x_{d_n} \le1+ \sum_{i=3}^n \log d_i^*. \qedhere
\]
\end{proof}

\begin{corollary}\label{cor:degen}
  Let\/ $G$ be a connected $d$-degenerate graph. If $d=2$, then $\EE[F] \le n-1$.
  If $d\ge 3$, then $\EE[F] \le 1 + (n-2) \log d$.
\end{corollary}

Theorem~\ref{thm:degbnd} improves upon the previous best known general bound, proven by Stahl in \cite{St91Short}. Letting $d'_1,d'_2,\ldots,d'_n$ denote the degree sequence of $G$, Stahl showed that the expected number of faces in a random embedding of $G$ is at most
\begin{equation}2n+\sum_{i=1}^{n} \log(d'_i).
\label{eq:stahlgn}
\end{equation}
Thus, our bound (\ref{eq:thm8-1}) yields an improvement by exchanging the term $2n$ with $1$
and replacing $d'_i$ with smaller values $d_i^*$ (when all $d'_i \ge 3$).
A similar improvement has been made with (\ref{eq:thm8-2}), which should be compared with the bound $\EE(F) \le n + \sum_{i=1}^{n}H_{d'_i -1}$ from \cite{St91Short}.

In a separate paper \cite{St91Linear}, Stahl described some infinite families of graphs for which $\EE(F)$ is linear in the number of vertices. Specifically, he showed, building the graph $G_n$ by linking together $n$ copies of a fixed graph $H$ ``in a consistent manner so as to form a chain,''  that the expected number of faces in a random embedding of $G_n$ is $\Theta(|V(G_n)|)$. In such \emph{linear families} of graphs, the maximum degree of $G_n$ is at most twice the maximum degree of $H$, itself an absolute constant with respect to $n$, so that the bound in Theorem~\ref{thm:degbnd} is tight up to a constant factor. 

At first sight, it may seem that the class of graphs satisfying $\EE(F) = \Theta(n)$ is very special. However, the following result shows that a wide variety of graphs have this property.

\begin{proposition}\label{prop:linear}
Let $G$ be a graph and let $\mathcal C$ be some family of cycles in $G$. Suppose that each cycle in $\mathcal C$ has length at most $\ell$ and all of its vertices have degree at most $d$. If $d\ne2$, then the expected number of faces in a random embedding of $G$ is at least $\frac{2 |\mathcal C|}{(d-1)^\ell}$. 
\label{thm:linearlb}
\end{proposition}

\begin{proof}
Consider a map $M$ of $G$ chosen uniformly at random. For each $C \in \mathcal{C}$ let $X_C$ be the indicator random variable for the event that ``$C$ is a facial walk of $M$.'' Notice that $F \ge \sum_{C \in \mathcal{C}} X_C$. Given a cycle $C=u_1u_2\dots u_k \in \mathcal{C}$ of length $k$, let $e_i = u_iu_{i+1}$ for $1 \leq i \leq k$, taking indices modulo $k$. Then $C$ is a face of $M$ if and only if one of the following holds: (1) $\pi_{u_i}(e_i)=e_{i-1}$ for $1 \leq i \leq k$ or (2) $\pi_{u_i}(e_{i-1}) = e_i$ for $1 \leq i \leq k$. By counting, it is easy to check that the probability that a random cyclic permutation of $\{1,2,\ldots,t\}$ sends 1 to 2 is $\frac{1}{t-1}$. Thus, as the rotation at each vertex is independent and it is impossible for both (1) and (2) to occur simultaneously (unless all vertices on $C$ are of degree 2), we have 
$$\EE(X_C) = 2 \prod_{i=1}^k \frac{1}{\deg(u_i)-1} \geq \frac{2}{(d-1)^\ell}.$$
The result now follows by using the linearity of expectation.
\end{proof}

Combining Theorem~\ref{thm:degbnd} with Proposition~\ref{thm:linearlb}, we see that any graph with bounded maximum degree and linearly many short cycles has linearly many expected faces. Although Proposition~\ref{prop:linear} describes general classes of graphs with linear expected number of faces, it is believed that this is rare.
In fact, Stahl conjectured \cite[Conjecture 4.3]{St90} that for almost all graphs with $q$ edges, the expected number of faces in a random embedding is close to $H_{2q}$.

We were unable to find any graph family with unbounded degeneracy for which the bound in Theorem~\ref{thm:degbnd} is tight. Indeed, we believe that such graphs do not exist and we propose a general conjecture that is given in the introduction as Conjecture \ref{conj:simple} (for simple graphs) and Conjecture \ref{endingconj} when we allow edges of high multiplicity.

\section*{Acknowledgements}

The authors would like to thank Amarpreet Rattan for pointing out relevant literature and Ladislav Stacho for helpful initial discussions on the topic.

\bibliographystyle{plainurl}
\bibliography{bibliography}

\begin{thebibliography}{10}

\bibitem{Andrews1994}
George~E. Andrews, David~M. Jackson, and Terry~I. Visentin.
\newblock A hypergeometric analysis of the genus series for a class of 2-cell
  embeddings in orientable surfaces.
\newblock {\em {SIAM} J. Math. Anal.}, 25(2):243--255, March 1994.
\newblock \href {https://doi.org/10.1137/s0036141092229549}
  {\path{doi:10.1137/s0036141092229549}}.

\bibitem{Ch11}
Guillaume Chapuy.
\newblock A new combinatorial identity for unicellular maps, via a direct
  bijective approach.
\newblock {\em Adv. in Appl. Math.}, 47(4):874--893, 2011.
\newblock \href {https://doi.org/10.1016/j.aam.2011.04.004}
  {\path{doi:10.1016/j.aam.2011.04.004}}.

\bibitem{Chen2020}
Ricky X.~F. Chen.
\newblock Combinatorially refine a {Z}agier-{S}tanley result on products of
  permutations.
\newblock {\em Discrete Math.}, 343(8):111912, 5, 2020.
\newblock \href {https://doi.org/10.1016/j.disc.2020.111912}
  {\path{doi:10.1016/j.disc.2020.111912}}.

\bibitem{CR16}
Ricky X.~F. Chen and Christian~M. Reidys.
\newblock Plane permutations and applications to a result of {Z}agier-{S}tanley
  and distances of permutations.
\newblock {\em SIAM J. Discrete Math.}, 30(3):1660--1684, 2016.
\newblock \href {https://doi.org/10.1137/15M1023646}
  {\path{doi:10.1137/15M1023646}}.

\bibitem{CMS12}
Robert Cori, Michel Marcus, and Gilles Schaeffer.
\newblock Odd permutations are nicer than even ones.
\newblock {\em European J. Combin.}, 33(7):1467--1478, 2012.
\newblock \href {https://doi.org/10.1016/j.ejc.2012.03.012}
  {\path{doi:10.1016/j.ejc.2012.03.012}}.

\bibitem{GS10}
Ian~P. Goulden and William Slofstra.
\newblock Annular embeddings of permutations for arbitrary genus.
\newblock {\em J. Combin. Theory Ser. A}, 117(3):272--288, 2010.
\newblock \href {https://doi.org/10.1016/j.jcta.2009.11.009}
  {\path{doi:10.1016/j.jcta.2009.11.009}}.

\bibitem{GMT17}
Jonathan~L. Gross, Toufik Mansour, and Thomas~W. Tucker.
\newblock Valence-partitioned genus polynomials and their application to
  generalized dipoles.
\newblock {\em Australas. J. Combin.}, 67:203--221, 2017.
\newblock URL: \url{https://ajc.maths.uq.edu.au/pdf/67/ajc_v67_p203.pdf}.

\bibitem{Gross89}
Jonathan~L. Gross, David~P. Robbins, and Thomas~W. Tucker.
\newblock Genus distributions for bouquets of circles.
\newblock {\em J. Combin. Theory Ser. B}, 47(3):292--306, 1989.
\newblock \href {https://doi.org/10.1016/0095-8956(89)90030-0}
  {\path{doi:10.1016/0095-8956(89)90030-0}}.

\bibitem{HZ86}
John Harer and Don Zagier.
\newblock The {E}uler characteristic of the moduli space of curves.
\newblock {\em Invent. Math.}, 85(3):457--485, 1986.
\newblock \href {https://doi.org/10.1007/BF01390325}
  {\path{doi:10.1007/BF01390325}}.

\bibitem{Ja87}
David~M. Jackson.
\newblock Counting cycles in permutations by group characters, with an
  application to a topological problem.
\newblock {\em Trans. Amer. Math. Soc.}, 299(2):785--801, 1987.
\newblock \href {https://doi.org/10.2307/2000524} {\path{doi:10.2307/2000524}}.

\bibitem{Jackson1994_algebraic}
David~M. Jackson.
\newblock Algebraic and analytic approaches for the genus series for 2-cell
  embeddings on orientable and nonorientable surfaces.
\newblock In {\em Formal Power Series and Algebraic Combinatorics}, pages
  115--132, 1994.

\bibitem{Jackson1994_integral}
David~M. Jackson.
\newblock On an integral representation for the genus series for $2$-cell
  embeddings.
\newblock {\em Trans. Amer. Math. Soc.}, 344(2):755--772, February 1994.
\newblock \href {https://doi.org/10.1090/s0002-9947-1994-1236224-5}
  {\path{doi:10.1090/s0002-9947-1994-1236224-5}}.

\bibitem{KL93}
Jin~Ho Kwak and Jaeun Lee.
\newblock Genus polynomials of dipoles.
\newblock {\em Kyungpook Math. J.}, 33(1):115--125, 1993.
\newblock URL:
  \url{https://www.koreascience.or.kr/article/JAKO199325748114657.page}.

\bibitem{LandoZvonkin}
Sergei~K. Lando and Alexander~K. Zvonkin.
\newblock {\em Graphs on surfaces and their applications}, volume 141 of {\em
  Encyclopaedia of Mathematical Sciences}.
\newblock Springer-Verlag, Berlin, 2004.
\newblock With an appendix by Don B. Zagier, Low-Dimensional Topology, II.
\newblock \href {https://doi.org/10.1007/978-3-540-38361-1}
  {\path{doi:10.1007/978-3-540-38361-1}}.

\bibitem{MT01}
Bojan Mohar and Carsten Thomassen.
\newblock {\em Graphs on surfaces}.
\newblock Johns Hopkins Studies in the Mathematical Sciences. Johns Hopkins
  University Press, Baltimore, MD, 2001.

\bibitem{NOdM12}
Jaroslav Ne\v{s}et\v{r}il and Patrice Ossona~de Mendez.
\newblock {\em Sparsity}, volume~28 of {\em Algorithms and Combinatorics}.
\newblock Springer, Heidelberg, 2012.
\newblock Graphs, structures, and algorithms.
\newblock \href {https://doi.org/10.1007/978-3-642-27875-4}
  {\path{doi:10.1007/978-3-642-27875-4}}.

\bibitem{RiThesis}
Robert~G. Rieper.
\newblock {\em The enumeration of graph imbeddings}.
\newblock PhD thesis, Western Michigan University, Kalamazoo, MI, 1987.

\bibitem{St90}
Saul Stahl.
\newblock Region distributions of graph embeddings and {S}tirling numbers.
\newblock {\em Discrete Math.}, 82(1):57--78, 1990.
\newblock \href {https://doi.org/10.1016/0012-365X(90)90045-J}
  {\path{doi:10.1016/0012-365X(90)90045-J}}.

\bibitem{St91Linear}
Saul Stahl.
\newblock Permutation-partition pairs. {III}. {E}mbedding distributions of
  linear families of graphs.
\newblock {\em J. Combin. Theory Ser. B}, 52(2):191--218, 1991.
\newblock \href {https://doi.org/10.1016/0095-8956(91)90062-O}
  {\path{doi:10.1016/0095-8956(91)90062-O}}.

\bibitem{St91Short}
Saul Stahl.
\newblock An upper bound for the average number of regions.
\newblock {\em J. Combin. Theory Ser. B}, 52(2):219--221, 1991.
\newblock \href {https://doi.org/10.1016/0095-8956(91)90063-P}
  {\path{doi:10.1016/0095-8956(91)90063-P}}.

\bibitem{Stahl1995}
Saul Stahl.
\newblock On the average genus of the random graph.
\newblock {\em J. Graph Theory}, 20(1):1--18, August 1995.
\newblock \href {https://doi.org/10.1002/jgt.3190200102}
  {\path{doi:10.1002/jgt.3190200102}}.

\bibitem{St11}
Richard~P. Stanley.
\newblock Two enumerative results on cycles of permutations.
\newblock {\em European J. Combin.}, 32(6):937--943, 2011.
\newblock \href {https://doi.org/10.1016/j.ejc.2011.01.011}
  {\path{doi:10.1016/j.ejc.2011.01.011}}.

\bibitem{Wh73}
Arthur~T. White.
\newblock {\em Graphs, groups and surfaces}.
\newblock North-Holland Publishing Co., Amsterdam-London; American Elsevier
  Publishing Co., Inc., New York, 1973.
\newblock North-Holland Mathematics Studies, No. 8.

\bibitem{Wh94}
Arthur~T. White.
\newblock An introduction to random topological graph theory.
\newblock {\em Combin. Probab. Comput.}, 3(4):545--555, 1994.
\newblock \href {https://doi.org/10.1017/S0963548300001395}
  {\path{doi:10.1017/S0963548300001395}}.

\bibitem{Za95}
Don Zagier.
\newblock On the distribution of the number of cycles of elements in symmetric
  groups.
\newblock {\em Nieuw Arch. Wisk. (4)}, 13(3):489--495, 1995.

\end{thebibliography}

\end{document}